\documentclass[11pt]{amsart}

\usepackage{amsmath,amsfonts,amsbsy,amsgen,amscd,amsxtra,amstext,amsopn,amssymb}
\usepackage{latexsym}
\usepackage{mathrsfs}
\usepackage{pictex}
\usepackage{float}
\usepackage{euscript}
\usepackage{graphicx}

\newtheorem{theorem}{Theorem}[section]
\newtheorem{proposition}[theorem]{Proposition}
\newtheorem{lemma}[theorem]{Lemma}
\newtheorem{corollary}[theorem]{Corollary}
\theoremstyle{definition}
\newtheorem{definition}[theorem]{Definition}

\theoremstyle{remark}
\newtheorem{remark}[theorem]{Remark}

\numberwithin{equation}{section}

\def\R{\mathbb R}
\def\C{\mathbb C}

\def\N{\mathbb N}
\def\indic{{\rm 1\kern -2.3pt I\hskip 1pt}}
\def\E{\mathop{\mathbb E}}
\def\a{\alpha}
\def\d{\delta}
\def\e{\varepsilon}

\def\s{\sigma}
\def\z{\zeta}
\def\<{\langle}
\def\>{\rangle}

\def\dpr{{d_{\proj}}}
\def\dsp{{d_{\bS}}}

\def\cM{{\mathscr M}}
\def\vol{\mathsf{vol}}

\def\Oh{\mathcal O}
\def\bd{\mathbf{d}}
\def\HH{{\mathcal H}}
\def\cE{{\mathcal E}}
\def\cC{{\mathcal C}}
\def\cR{{\mathcal R}}
\def\Hd{{{\mathcal H}_{\bd}}}
\def\HdR{{\mathcal H}_{\bd}^{\R}}
\def\proj{{\mathbb P}}
\def\mun{\mu_{\rm norm}}
\def\mum{\mu_{\max}}
\def\size{\mathrm{size}}
\def\av{\mathrm{av}}
\def\Prob{\mathop{\mathrm{Prob}}}

\def\bS{\mathbb{S}}

\def\Dn{\mathcal{D}}
\def\of{\overline{f}}
\def\og{\overline{g}}
\def\oh{\overline{h}}
\def\oq{\overline{q}}
\def\ok{\overline{k}}
\def\oA{\overline{A}}

\def\ox{\overline{x}}
\def\NJ{\mathrm{NJ}}

\def\oM{\overline{M}}

\def\diag{\mathsf{diag}}

\def\NJ{\mathrm{NJ}}
\def\ox{\overline{x}}

\def\graph{\mathrm{graph}}

\def\cU{\mathcal{U}}
\def\Re{\mathrm{Re}}
\def\id{\mathrm{id}}
\def\Id{\mathrm{I}}
\def\H{{\mathcal H}}
\def\op{{}}

\def\spann{\mathsf{span}}
\def\rk{\mathsf{rank}}
\def\Var{\mathsf{Var}}
\def\time{\mathsf{time}}
\def\ALH{{\sf{ALH}}}
\def\LV{{\sf{LV}}}
\def\FD{{\sf{MD}}}
\renewcommand{\bar}{\overline}

\renewcommand{\tilde}{\widetilde}

\renewcommand{\Re}{\mathsf{Re}}
\def\oU{\bar{U}}
\def\bz{{\boldsymbol z}}

\def\rhozero{\rho_{\bf st}}

\def\aff{\mathrm{aff}}
\def\ReA{{\sf{Ren}}}
\def\Prf{\mathrm{Probfail}}
\def\IRe{{\sf{ItRen}}}

\def\hV{V}
\def\hW{W}
\def\Vp{V_{\proj}}
\def\ps{p}
\newcommand{\transp}{^{\rm T}}
\def\HdR{{\mathcal H}_{\bd,\R}}
\def\pr{\mathrm{pr}}
\def\algo{\begin{center}
               \begin{minipage}{6in}
               \begin{tabbing}
               \marks}
\def\falgo{\end{tabbing}
                \end{minipage}
                \end{center}}
\def\marks{nn\= nn\= nn\= nn\= nn\= nn\= nn\= \kill}

\def\eproof{{\mbox{}\hfill\qqed}\medskip}
\newcommand\qqed{{\unskip\nobreak\hfil\penalty50\hskip2em\vadjust{}
\nobreak\hfil$\Box$\parfillskip=0pt\finalhyphendemerits=0\par}}

\begin{document}

\title[On a Problem Posed by Steve Smale]
{\bf On a Problem Posed by Steve Smale}

\author{Peter B\"urgisser}
\thanks{Partially supported by DFG grant BU 1371/2-1 and
Paderborn Institute for Scientific Computation (PaSCo).}

\address{Institute of Mathematics, University of Paderborn, D-33098 Paderborn, Germany}
\email{pbuerg@upb.de}

\author{Felipe Cucker}
\thanks{Partially supported by GRF grant CityU 100810.}

\address{Dept.\ of Mathematics, City University of Hong Kong,
Kowloon Tong, Hong Kong}
\email{macucker@cityu.edu.hk}

\thanks{An extended abstract of this work was presented at STOC 2010 under the title
``Solving Polynomial Equations in Smoothed Polynomial Time and a Near Solution to Smale's
17th Problem''}

\date{\today}

\begin{abstract}
The 17th of the problems proposed by Steve Smale for the 21st century
asks for the existence of a deterministic algorithm
computing an approximate solution of a system of $n$ complex polynomials
in $n$ unknowns in time polynomial, on the average, in the size $N$ of the
input system. A partial solution to this problem was given by Carlos Beltr\'an and
Luis Miguel Pardo who exhibited a randomized algorithm doing so. In this
paper we further extend this result in several directions. Firstly, we exhibit a
linear homotopy algorithm that efficiently implements a non-constructive
idea of Mike Shub. This algorithm is then used in a randomized algorithm,
call it \LV, \`a la Beltr\'an-Pardo.  Secondly, we
perform a smoothed analysis (in the sense of Spielman and Teng) of
algorithm \LV\ and prove that its
smoothed complexity is polynomial in the input size and $\s^{-1}$,
where $\s$ controls the size of of the random perturbation of the input
systems. Thirdly, we  perform a condition-based analysis of  \LV. That is,
we give a bound, for each system $f$, of the expected running time
of \LV\ with input $f$. In addition to its dependence on $N$ this bound
also depends on the condition of $f$. Fourthly, and to conclude, we return
to Smale's 17th problem as originally formulated for deterministic algorithms.
We exhibit such an algorithm and show that its average complexity is
$N^{\Oh(\log\log N)}$. This is nearly a solution to Smale's 17th problem.
\end{abstract}



\maketitle

{\small
\tableofcontents
}

\section{Introduction}

In 2000, Steve Smale published a list of mathematical problems for the 21st
century~\cite{smale:00}. The 17th problem in the list reads as follows:
{\small
\begin{quote}
{\em Can a zero of $n$ complex polynomial equations in $n$ unknowns be found
approximately, on the average, in polynomial time with a uniform algorithm?}
\end{quote}
}
Smale pointed out that ``it is reasonable'' to homogenize the polynomial
equations by adding a new variable and to work in projective space after which
he made precise the different notions intervening in the question above. We
provide these definitions in full detail in Section~\ref{sec:prelim}.
Before doing so, in the
remainder of this section, we briefly describe the recent history of Smale's 17th
problem and the particular contribution of the present paper. The following
summary of notations should suffice for this purpose.

We denote by $\Hd$ the linear space of complex homogeneous polynomial
systems in $n+1$ variables, with a fixed degree pattern
$\bd=(d_1,\ldots,d_n)$. We let $D=\max_i d_i$, $N=\dim_{\C}\Hd$, and
$\Dn=\prod_i d_i$. We endow this space with the unitarily invariant
Bombieri-Weyl Hermitian product and consider the unit sphere $S(\Hd)$
with respect to the norm induced by this product. We then make this
sphere a probability space by considering the uniform measure on
it. The expression ``on the average'' refers to expectation on
this probability space. Also, the expression ``approximate zero''
refers to a point for which Newton's method, starting at it, converges
immediately, quadratically fast.

This is the setting underlying the series of
papers~\cite{Bez1,Bez2,Bez3,Bez4,Bez5} ---commonly referred to as ``the
B\'ezout series''--- written by Shub and Smale
during the first half of the 1990s, a collection of ideas, methods, and results
that pervade all the research done in Smale's 17th problem since this was
proposed. The overall idea in the B\'ezout series is to use a linear homotopy.
That is, one starts
with a system $g$ and a zero $\z$ of $g$ and considers the segment
$E_{f,g}$ with extremities $f$ and $g$. Here $f$ is the system whose zero
we want to compute. Almost surely, when one moves from $g$ to $f$,
the zero $\z$ of $g$ follows a curve in projective space to end in a zero
of $f$. The homotopy method consists of dividing the segment $E_{f,g}$ in
a number, say $k$, of subsegments $E_i$ small enough to ensure that an
approximate zero $x_i$ of the system at the origin of $E_i$ can be made into
an approximate zero $x_{i+1}$ of the system at its end (via one step of
Newton's method).
The difficulty of this overall idea lies in the following issues:
\begin{enumerate}
\item
How does one choose the initial pair $(g,\z)$?
\item
How does one choose the subsegments $E_i$? In particular, how
large should~$k$ be?
\end{enumerate}
The state of the art at the end of the B\'ezout series, i.e., in~\cite{Bez5},
showed an incomplete picture. For~(2), the rule consisted of taking a
regular subdivision of $E_{f,g}$ for a given $k$, executing the path-following
procedure, and repeating with $k$ replaced by $2k$  if the final point could not
be shown to be an approximate zero of $f$ (Shub and Smale provided criteria
for checking this). Concerning~(1), Shub and Smale proved that good initial pairs
$(g,\z)$ (in the sense that the average number of iterations for the rule above
was polynomial in the size of $f$) existed for each degree pattern $\bd$, but
they could not exhibit a procedure to  generate one such pair.

The next breakthrough took a decade to come. Beltr\'an and Pardo proposed
in~\cite{BeltranPardo08,BePa08a} that the initial pair $(g,\z)$
should be randomly chosen. The consideration of randomized
algorithms departs from the formulation
of Smale's 17th problem\footnote{In his description of Problem~17 Smale
writes ``Time is measured by the number of arithmetic operations and comparisons,
$\leq$, using real machines (as in Problem~3)'' and in the latter he points that,
``In [Blum-Shub-Smale,1989] a satisfactory definition [of these machines] is
proposed.'' The paper~\cite{bss:89} quoted by Smale deals exclusively with
deterministic machines. Furthermore, Smale adds that ``a probability measure
must be put on the space of all such~$f$, for each $\bd=(d_1,\ldots,d_n)$, and
the time of an algorithm is averaged over the space of $f$.'' Hence, the
expression  `average time' refers to expectation over the input data only.}
but it is widely accepted that, in practical terms,
such algorithms are as good as their deterministic siblings. And in the case
at hand this departure turned out to pay off. The average (over $f$) of the
expected (over $(g,\z)$) number of iterations of the algorithm
proposed in~\cite{BePa08a} is $\Oh(n^5N^2D^3\log\Dn)$.
One of the most notable features of the ideas introduced
by Beltr\'an and Pardo is the use of a measure on the space
of pairs $(g,\z)$ which is friendly enough to perform a probabilistic
analysis while, at the same time, does allow for efficient sampling.

Shortly after the publication of~\cite{BeltranPardo08,BePa08a}
Shub wrote a short paper of great importance~\cite{Bez6}. Complexity
bounds in both the B\'ezout series and the Beltr\'an-Pardo
results rely on condition numbers. Shub and Smale had introduced
a measure of condition $\mun(f,\z)$ for $f\in\Hd$ and
$\z\in\C^{n+1}$ which, in case  $\z$ is a zero of $f$, quantifies
how much $\z$ varies when $f$ is slightly perturbed. Using
this measure they defined the {\em condition number} of a system
$f$ by taking
\begin{equation}\label{eq:muSS}
     \mum(f):=\max_{\z\mid f(\z)=0}\mun(f,\z).
\end{equation}
The bounds mentioned above make use of an estimate for the
worst-condi\-tioned system along the segment $E_{f,g}$, that is,
of the quantity
\begin{equation}\label{eq:mumax}
     \max_{q\in E_{f,g}}\mum(q).
\end{equation}
The main result in~\cite{Bez6} shows that there exists a partition of
$E_{f,g}$ which successfully computes an approximate zero of $f$
whose number $k$ of pieces satisfies
\begin{equation}\label{eq:integral_mu2}
     k\leq C D^{3/2}\int_{q\in E_{f,g}} \mu_2^2(q)\, dq ,
\end{equation}
where $C$ is a constant and $\mu_2(q)$ is the
{\em mean square condition number}
of $q$ given by
\begin{equation}\label{eq:mu2}
   \mu_2^2(q):=\frac{1}{\Dn}\sum_{\z\mid q(\z)=0}\mun^2(q,\z).
\end{equation}
This partition is explicitly described in~\cite{Bez6},
but no constructive procedure to compute the partition is given there.

In an oversight of this non-constructibility, Beltr\'an and
Pardo~\cite{BePa08b} provided a new version of their
randomized algorithm\footnote{The algorithm in~\cite{BePa08b} explicitly
calls as a subroutine ``the homotopy algorithm of~\cite{Bez6}'' without noticing
that the partition
in~\cite{Bez6} is non-algorithmic. Actually, the word `algorithm' is never used
in~\cite{Bez6}. The main goal of~\cite{Bez6}, as stated in the abstract, is to
motivate ``the study of short paths or geodesics in the condition metric''
---the proof of~\eqref{eq:integral_mu2} does not require the homotopy to be
linear and one may wonder whether other paths in $\Hd$ may substantially
decrease the integral in the right-hand side.
This goal has been addressed, but not attained, in~\cite{Bez7}. As of today
it remains a fascinating
open problem.} with an improved complexity of $\Oh(D^{3/2}nN)$.
\smallskip

A first goal of this paper is to validate Beltr\'an and Pardo's analysis
in~\cite{BePa08b} by exhibiting an efficiently constructible partition
of $E_{f,g}$ which satisfies a bound like~\eqref{eq:integral_mu2}. Our
way of doing so owes much to the ideas in~\cite{Bez6}. The path-following
procedure \ALH\ relying on this partition  is described in detail
in~\S\ref{se:homotopy} together with a result,
Theorem~\ref{thm:main_path_following}, bounding its complexity
as in~\eqref{eq:integral_mu2}.

The second goal of this paper is to perform a smoothed analysis  of
a randomized algorithm (essentially Beltr\'an-Pardo randomization plus \ALH)
computing a zero of $f$, which we call \LV.
What smoothed analysis is, is succinctly explained
in the citation of the G\"odel prize 2008 awarded to its creators,
Daniel Spielman and Teng Shang-Hua\footnote{See~%
{\tt http://www.fmi.uni-stuttgart.de/ti/personen/Diekert/citation08.pdf}
for the whole citation}.
{\small
\begin{quote}
Smoothed Analysis is a novel approach to the analysis of algorithms. It
bridges the gap between worst-case and average case behavior by considering
the performance of algorithms under a small perturbation of the input.
As a result, it provides a new rigorous framework for explaining the practical
success of algorithms and heuristics that could not be well understood
through traditional algorithm analysis methods.
\end{quote}
}
In a nutshell, smoothed analysis is a probabilistic analysis which replaces the
`evenly spread' measures underlying the usual average-case analysis
(uniform measures, standard normals, \dots) by a measure centered at the
input data. That is, it replaces the `average data input'
(an unlikely input in actual computations) by a small random
perturbation of a worst-case data and substitutes the typical
quantity studied in the average-case context,
$$
     \E_{f\sim \cR} \varphi(f),
$$
by
$$
   \sup_{\of} \E_{f\sim \cC(\of,r)} \varphi(f).
$$
Here $\varphi(f)$ is the function of $f$ one is interested in (e.g., the complexity
of an algorithm over input $f$),
$\cR$ is the `evenly spread' measure mentioned above and
$\cC(\of,r)$ is an isotropic measure centered at $\of$ with a dispersion
(e.g., variance) given by a (small) parameter $r>0$.

An immediate advantage of smoothed analysis is its robustness
with respect to the measure $\cC$ (see~\S\ref{subsec:SA} below).
This is in contrast with the most common critique to average-case
analysis: ``A bound on the performance of an algorithm under one
distribution says little about its performance under another distribution,
and may say little about the inputs that occur in practice''~\cite{ST:02}.

The precise details of the smoothed analysis we perform for
zero finding are in~\S\ref{subsec:SA}.

To describe the third goal of this paper we recall Smale's ideas
of complexity analysis as exposed in~\cite{Smale97}. In this
program-setting paper Smale writes that he sees ``much of the
complexity theory [\dots] of numerical analysis conveniently
represented by a two-part scheme.'' The first part amounts to
obtain, for the running time $\time(f)$ of an algorithm on input $f$,
an estimate of the form
\begin{equation}\label{eq:smale1}
    \time(f)\leq K(\size(f) +\mu(f))^c
\end{equation}
where $K,c$ are positive constants and $\mu(f)$ is a condition number
for $f$. The second takes the form
\begin{equation}\label{eq:smale2}
    \Prob\{\mu(f)\geq T\}\ \leq\  T^{-c}
\end{equation}
``where a probability measure has been put on the space of inputs.''
The first part of this scheme provides understanding on the behavior
of the algorithm for specific inputs $f$ (in terms of their condition as
measured by $\mu(f)$). The second, combined with the first, allows
one to obtain probability bounds for $\time(f)$ in terms of $\size(f)$
only. But these bounds say little about $\time(f)$ for actual input
data $f$.

Part one of Smale's program is missing in the work related with his
17th problem. All estimates on the running time of path-following
procedures for a given $f$ occurring in both the B\'ezout series
and the work by Beltr\'an and Pardo are expressed in terms of
the quantity in~\eqref{eq:mumax} or the integral
in~\eqref{eq:integral_mu2}, not purely in terms of the condition of
$f$. We fill this gap by showing for the expected running time
of \LV\ a bound like~\eqref{eq:smale1}
with $\mu(f)=\mum(f)$. The precise statement,
Theorem~\ref{thm:CBA}, is in~\S\ref{subsec:CBA} below.

Last but not least, to close this introduction, we return to its opening theme:
Smale's 17th problem. Even though randomized algorithms are
efficient in theory and reliable in practice they do not offer an answer to
the question of the existence of
a deterministic algorithm computing approximate zeros of complex polynomial
systems in average polynomial time.  The situation is akin to the development of
primality testing. It was precisely with this problem that randomized algorithms
became a means to deal with apparently intractable
problems~\cite{SoSt:77,Rabin:80}. Yet, the eventual display of a
deterministic polynomial-time
algorithm~\cite{AKS:04} was justly welcomed as a major achievement.
The fourth main result in this paper exhibits a deterministic algorithm computing
approximate zeros in average time $N^{\Oh(\log\log N)}$.
To do so we design and analyze
a deterministic homotopy algorithm, call it \FD, whose average complexity is
polynomial in $n$ and $N$ and exponential in $D$. This already yields
a polynomial-time algorithm when one restricts the degree $D$ to be
at most $n^{1-\e}$ for any fixed $\e>0$ (and, in particular, when $D$ is
fixed as in a system of quadratic or cubic equations). Algorithm  \FD\ is
fast when $D$ is small.
We complement it with an algorithm that uses a procedure proposed
by Jim Renegar~\cite{rene:89} and which
computes approximate zeros similarly fast when $D$ is large.

In order to prove the results described above we have relied
on a number of ideas and techniques. Some of them ---e.g., the
use of the coarea formula or of the Bombieri-Weyl Hermitian inner product---
are taken from the B\'ezout series and are pervasive in the literature on the
subject. Some others ---notably the use of the Gaussian distribution
and its truncations in Euclidean space instead of the uniform distribution
on a sphere or a projective space--- are less common. The blending of
these ideas has allowed us a development which unifies the treatment
of the several situations we consider for zero finding in this paper.


\subsection*{Acknowledgments}

Thanks go to Carlos Beltr\'an and Jean-Pierre Dedieu
for helpful comments.
We are very grateful to Mike Shub for constructive criticism and
insightful comments that helped to improve the paper
considerably.
This work was finalized during the special semester on Foundations
of Computational Mathematics in the fall of 2009.
We thank the Fields Institute in Toronto for hospitality and
financial support.

\section{Preliminaries}\label{sec:prelim}

\subsection{Setting and Notation}

For $d\in\N$ we denote by $\HH_d$ the subspace of $\C[X_0,\ldots,X_n]$
of homogeneous polynomials of degree $d$. For $f\in \HH_d$ we write
$$
   f(x) = \sum_{\alpha} {d \choose \alpha}^{1/2}\, a_\alpha X^\alpha
$$
where $\alpha=(\alpha_0, \dots, \alpha_n)$ is assumed to range over
all multi-indices such that $|\alpha| = \sum_{k=0}^n \alpha_k = d$,
${d\choose \alpha}$ denotes the multinomial coefficient,
and $X^\alpha:= X_0^{\alpha_0}X_1^{\alpha_1}\cdots X_n^{\alpha_n}$.
That is, we take for basis of the linear space $\HH_d$ the
{\em Bombieri-Weyl} basis consisting of the monomials
${d \choose \alpha}^{1/2}X^\alpha$. A reason to do so is
that the Hermitian inner product associated to this basis is unitarily invariant.
That is, if $g\in\HH_d$ is given by
$g(x) = \sum_{\alpha} {d\choose \alpha}^{1/2}b_\alpha X^\alpha$,
then the canonical Hermitian inner product
$$
    \langle f,g\rangle =\sum_{|\alpha|=d} a_\alpha\, \overline{b_\alpha}
$$
satisfies, for all element $\nu$ in the unitary group $\cU(n+1)$, that
$$
    \langle f,g\rangle =\langle f\circ \nu,g\circ\nu\rangle.
$$
Fix $d_1,\ldots,d_n\in\N\setminus\{0\}$ and
let $\Hd=\HH_{d_1}\times \ldots\times\HH_{d_n}$ be the
vector space of polynomial systems $f=(f_1,\ldots,f_n)$ with
$f_i\in\C[X_0,\ldots,X_n]$ homogeneous of degree $d_i$.
The space $\Hd$ is naturally endowed with a Hermitian inner product
$\langle f,g \rangle = \sum_{i=1}^n \langle f_i, g_i \rangle$.
We denote by $\|f\|$ the corresponding norm of $f\in\Hd$.

Recall that $N=\dim_\C\Hd$ and $D=\max_i d_i$.
Also, in the rest of this paper, we
assume $D\geq 2$ (the case $D=1$ being solvable with elementary
linear algebra).

Let $\proj^n:=\proj(\C^{n+1})$ denote the complex projective space
associated to $\C^{n+1}$ and $S(\Hd)$ the unit sphere of $\Hd$.
These are smooth manifolds that naturally carry the structure of a
Riemannian manifold
(for $\proj^n$ the metric is called Fubini-Study metric).
We will denote by $\dpr$ and $\dsp$ their Riemannian distances which,
in both cases, amount to the angle between the arguments.
Specifically, for $x,y\in\proj^n$ one has
\begin{equation}\label{eq:projdist}
 \cos\dpr(x,y) = \frac{|\langle x,y\rangle|}{\|x\|\, \|y\|}.
\end{equation}
Ocasionally, for
$f,g\in\Hd\setminus\{0\}$, we will abuse language and write
$\dsp(f,g)$ to denote this angle,
that is, the distance $\dsp\big(\frac{f}{\|f\|},\frac{g}{\|g\|}\big)$.

We define the {\em solution variety} to be
$$
  \Vp:= \{(f,\z)\in \Hd\times \proj^n \mid f\neq0 \mbox{ and } f(\z)= 0\}.
$$
This is a smooth submanifold of $\Hd\times \proj^n$
and hence also carries a Riemannian structure.
We denote by $\Vp(f)$ the zero set of $f\in\Hd$ in $\proj^n$.
By B\'ezout's Theorem, it contains $\Dn$ points for almost all $f$.
Let $Df(\zeta)_{|T_\zeta}$ denote the restriction of the derivative of
$f\colon\C^{n+1}\to\C^n$ at~$\zeta$ to the tangent space
$T_\zeta :=\{v\in\C^{n+1}\mid \langle v,\zeta\rangle = 0\}$
of $\proj^n$ at $\zeta$.
The {\em subvariety of ill-posed pairs} is defined as
\begin{equation*}\label{eq:Sigma'}
 \Sigma'_\proj := \{(f,\z)\in \Vp\mid \rk\,Df(\zeta)_{|T_\zeta} < n\} .
\end{equation*}
Note that $(f,\z)\not\in\Sigma'_\proj$ means that $\z$ is a simple zero of~$f$.
In this case, by the implicit function theorem, the projection
$\Vp\to\Hd,(g,x)\mapsto g$
can be locally inverted around $(f,\z)$.
The image $\Sigma$ of $\Sigma'_\proj$ under the
projection $\Vp\to\Hd$ is called the
{\em discriminant variety}.

\subsection{Newton's Method}\label{se:Newton}

In~\cite{Shub93b}, Mike Shub introduced the following projective version
of Newton's method.
We associate to $f\in\Hd$
(with $Df(x)$ of rank~$n$ for some~$x$) a map
$N_f:\C^{n+1}\setminus\{0\}\to\C^{n+1}\setminus\{0\}$ defined
(almost everywhere) by
$$
       N_f(x)=x-Df(x)_{|T_x}^{-1}f(x).
$$
Note that $N_f(x)$ is homogeneous of degree~0 in~$f$ and
of degree~1 in~$x$ so that
$N_f$ induces a rational map from $\proj^n$ to $\proj^n$
(which we will still denote by $N_f$) and this map
is invariant under multiplication of $f$ by constants.

We note that $N_f(x)$ can be  computed from $f$  and $x$
very efficiently: since the Jacobian $Df(x)$ can be evaluated with $\Oh(N)$
arithmetic operations~\cite{bast:83}, one can do with a total of
$\Oh(N+n^3)$ arithmetic operations.

It is well-known that when $x$ is sufficiently close to a simple zero $\z$ of $f$,
the sequence of Newton iterates beginning at $x$ will converge quadratically
fast to $\z$. This property lead Steve Smale to define the following intrinsic
notion of approximate zero.

\begin{definition}\label{def:app-zero}
By an {\em approximate zero} of $f\in\Hd$ associated with a zero $\z\in\proj^n$ of $f$
we understand a point $x\in\proj^n$ such that the sequence of Newton iterates
(adapted to projective space)
$$
   x_{i+1}:= N_{f}(x_{i})
$$
with initial point $x_0:=x$ converges immediately quadratically to $\z$, i.e.,
$$
 \dpr(x_i,\z)\le \Big(\frac12\Big)^{2^i -1}\ \dpr(x_0,\z)
$$
for all $i\in\N$.
\end{definition}

\subsection{Condition Numbers}\label{se:cond-num}

How close need $x$ to be from $\z$ to be an approximate zero? This
depends on how well conditioned the zero $\z$ is.

For $f\in\Hd$ and $x\in\C^{n+1}\setminus\{0\}$ we define the {\em
(normalized) condition number}~$\mun(f,x)$ by
$$
  \mun(f,x):= \|f\|\left\|  \big(Df(x)_{\mid T_x}\big)^{-1}
  \diag(\sqrt{d_1}\|x\|^{d_1-1},\dots,\sqrt{d_n}
   \|x\|^{d_n-1})\right\|_\op ,
$$
where $T_x$ denotes the Hermitian complement of $\C x$,
the right-hand side norm is the spectral norm,
and $\diag(a_i)$ denotes the diagonal matrix with entries~$a_i$.
Note that $\mun(f,x)$ is homogeneous of degree 0 in both arguments,
hence it is well defined for $(f,x)\in \Hd\times \proj^n$.
If $x$ is a simple zero of~$f$, then
$\ker Df(x)=\C x$ and hence
$\big(Df(x)_{\mid T_x}\big)^{-1}$ can be identified with
the Moore-Penrose inverse $Df(x)^\dagger$ of~$Df(x)$.
We have $\mun(f,x)\ge 1$,
cf.~\cite[\S12.4, Cor.~3]{bcss:95}.

The following result (essentially, a $\gamma$-Theorem in Smale's
theory of estimates for Newton's method~\cite{Smale86}) quantifies
our claim above.

\begin{theorem}\label{thm:gamma}
Assume $f(\zeta)=0$ and $\dpr(x,\zeta)\leq
\frac{u_0}{D^{3/2}\mun(f,\zeta)}$ where
$u_0:=3-\sqrt{7}\approx 0.3542$.
Then $x$ is an approximate zero of $f$ associated with $\z$.
\end{theorem}

\begin{proof}
This is an immediate consequence of the projective $\gamma$-Theorem
in~\cite[p.~263, Thm.~1]{bcss:95} combined with the higher
derivative estimate~\cite[p.~267, Thm.~2]{bcss:95}.
\end{proof}

\subsection{Gaussian distributions}\label{se:Gauss}

The distribution of input data will be modelled with Gaussians.
Let $\ox\in\R^n$ and $\s>0$.
We recall that the Gaussian distribution $N(\ox,\s^2\Id)$ on $\R^n$
with mean~$\ox$ and covariance matrix $\s^2\Id$
is given by the density
$$
 \rho(x) = \Big(\frac1{\s\sqrt{2\pi}}\Big)^n\,
  \exp\big(-\frac{\|x -\ox\|^2}{2\s^2}\big) .
$$

\section{Statement of Main Results}\label{sec:main_results}

\subsection{The Homotopy Continuation Routine \ALH}\label{se:homotopy}

Suppose that we are given an input system $f\in\Hd$ and an initial
pair $(g,\z)$ in the solution variety~$\Vp$ such that
$f$ and $g$ are $\R$-linearly independent.
Let $\a=\dsp(f,g)$. Consider the line segment
$E_{f,g}$ in $\Hd$ with endpoints~$f$ and~$g$. We parameterize this
segment by writing
$$
  E_{f,g}=\{q_\tau\in\Hd\mid \tau\in[0,1]\}
$$
with $q_\tau$ being the only point in $E_{f,g}$ such that
$\dsp(g,q_\tau)=\tau\a$ (see Figure~\ref{fig:1}).
Explicitly, we have
$q_\tau = tf+(1-t)g$, where $t=t(\tau)$ is given by
Equation~\eqref{eq:lambda} below.
If $E_{f,g}$ does not intersect the discriminant variety~$\Sigma$,
there is a unique continuous map
$[0,1]\to \Vp,\tau \mapsto (q_\tau,\z_\tau)$
such that $(q_0,\z_0)= (g,\z)$,
called the {\em lifting} of $E_{f,g}$
with origin $(g,\z)$.
In order to find an approximation
of the zero $\z_1$ of $f=q_1$ we may start with the zero $\z=\z_0$
of $g=q_0$ and numerically follow the path $(q_\tau,\z_\tau)$ by
subdividing $[0,1]$ with points $0=\tau_0<\tau_1<\cdots<\tau_k=1$
and by successively computing approximations~$x_i$ of~$\z_{\tau_i}$
by Newton's method.

More precisely, we consider the following algorithm \ALH\
(Adaptive Linear Homotopy)
with the stepsize parameter $\lambda=6.67\cdot 10^{-3}$.
\medskip

\begin{center}
\algo
\> Algorithm \ALH\\[2pt]
\>{\bf input} $f,g\in \Hd$ and $\zeta\in\proj^n$ such that $g(\zeta)=0$\\[2pt]
\>\> $\a:=\dsp(f,g)$, $r:=\|f\|$, $s:=\|g\|$\\[2pt]
\>\>$\tau:=0$, $q:=g$, $x:=\zeta$\\[2pt]
\>\> {\tt repeat}\\[2pt]
\>\>\>$\Delta\tau:= \frac{\lambda}{\alpha D^{3/2}\mun^2(q,x)}$\\[2pt]
\>\>\>$\tau:=\min\{1,\tau+\Delta\tau\}$\\[2pt]
\>\>\>$t:=\frac{s}{r\sin\a\cot(\tau\alpha) - r\cos\alpha + s}$\\[2pt]
\>\>\>$q:=tf+(1-t)g$\\[2pt]
\>\>\>$x:=N_q(x)$\\[2pt]
\>\> {\tt until} $\tau= 1$\\[2pt]
\>\> {\tt RETURN $x$}
\falgo
\end{center}

\medskip

Our main result for this algorithm, which we will prove in
Section~\ref{sec:ALH}, is the following.

\begin{theorem}\label{thm:main_path_following}
The algorithm \ALH\ stops after at most $k$ steps with
$$
     k \le 245\,D^{3/2}\,\dsp(f,g)\int_0^1\mun^2(q_\tau,\zeta_\tau)\,d\tau.
$$
The returned point $x$ is an approximate zero of $f$ with
associated zero $\z_1$.
\end{theorem}

\begin{remark}
{\bf 1.} The bound in Theorem~\ref{thm:main_path_following}
is optimal up to a constant factor. This easily follows
by an inspection of its proof given in \S\ref{sec:ALH}.

{\bf 2.} Algorithm \ALH\ requires the computation of $\mun$ which,
in turn, requires the computation of the operator norm of a matrix.
This cannot be done exactly with rational operations and square roots only.
We can do, however, with a sufficiently good approximation
of $\mun^2(q,x)$ and there exist several
numerical methods efficiently computing such an approximation. We will
therefore neglect this issue pointing, however, for the sceptical reader
that another course of action is possible. Indeed, one may replace
the operator by the Frobenius norm in the definition of $\mun$
and use the bounds
$\|M\|_\op \leq \|M\|_F\leq \sqrt{\rk(M)}\|M\|_\op$
to show that  this change
preserves the correctness of \ALH\ and adds a multiplicative factor $n$
in the right-hand side of Theorem~\ref{thm:main_path_following}.
A similar comment applies to the computation of~$\a$ and
$\cot(\tau\a)$ in algorithm \ALH\
which cannot be done exactly with rational operations.
\end{remark}

\subsection{Randomization and Complexity: the Algorithm \LV}\label{sec:R&C}

\ALH\ will serve as the basic routine for a number of algorithms
computing zeros of polynomial systems in different contexts. In
these contexts both the input system $f$ and the origin
$(g,\z)$ of the homotopy may be randomly chosen: in the case
of $(g,\z)$ as a computational technique and in the case of
$f$ in order to perform a probabilistic analysis of the algorithm's
running time.

In both cases, a probability measure is needed: one for
$f$ and one for the pair $(g,\z)$.
The measure for $f$ will depend on the kind of probabilistic ana\-lysis
(standard average-case or smoothed analysis) we perform. In contrast,
we will consider only one measure on~$\Vp$---which we denote by $\rhozero$---
for the initial pair $(g,\zeta)$.
It consists of drawing $g$ from $\Hd$ from the standard Gaussian distribution
(defined via the isomorphism $\Hd\simeq\R^{2N}$ given by the Bombieri-Weyl basis)
and then choosing one of the (almost surely) $\Dn$
zeros of~$g$ from the uniform distribution on $\{1,\ldots,\Dn\}$.
The formula for the density of $\rhozero$ will be derived later, see
Lemma~\ref{le:rho1}(5).
The above procedure is clearly non-constructive
as computing a zero of a system is the problem we wanted to solve
in the first place. One of the major contributions
in~\cite{BeltranPardo08} was to show that this drawback can be repaired. The
following result (a detailed version of the effective sampling
in~\cite{BePa08b}) will be proved in Section~\ref{se:eff-sample} as a special
case of more general results we will need in our development.

\begin{proposition}\label{prop:rho2}
We can compute a random pair $(g,\z)\in \Vp$ according to the density
$\rhozero$ with $\Oh(N)$ choices of random real numbers from the
standard Gaussian distribution and
$\Oh(DnN + n^3)$ arithmetic operations (including square roots of positive numbers).
\end{proposition}

Algorithms using randomly drawn data are called probabilistic (or
randomized). Those that always return a correct output are said to
be of type {\em Las Vegas}. The following algorithm (which uses
Proposition~\ref{prop:rho2}) belongs to this class. \smallskip

\algo
\> Algorithm \LV\\[2pt]
\>{\bf input} $f\in\Hd$\\[2pt]
\>\> {\tt draw $(g,\zeta)\in \Vp$ from $\rhozero$} \\[2pt]
\>\> {\tt run \ALH\ on input $(f,g,\z)$} \falgo
\medskip

For an input $f\in\Hd$ algorithm \LV\ either outputs an
approximate zero~$x$ of~$f$ or loops forever. By the {\em running
time} $t(f,g,\z)$ we will understand the number of elementary
operations (i.e., arithmetic operations,
evaluations of the elementary functions $\sin$, $\cos$, $\cot$, square root, and comparisons)
performed by \LV\ on input $f$ with initial pair
$(g,\z)$. For fixed~$f$, this is a random variable and its
expectation $t(f):=\E_{(g,\z)\sim\rhozero}(t(f,g,\z))$ is said to be
the {\em expected running time} of \LV\ on input~$f$.

For all $f,g,\z$, the running time $t(f,g,\z)$ is given by the
{\em number of iterations} $K(f,g,\z)$ of \ALH\ with input this triple times
the cost of an iteration, the latter being dominated by that of
computing one Newton iterate (which is $\Oh(N+n^3)$
independently of the triple $(f,g,\z)$, see \S\ref{se:Newton}).
It therefore follows that analyzing the expected running times
of \LV\ amounts to do so for the
expected value ---over $(g,\z)\in \Vp$ drawn from $\rhozero$--- of
$K(f,g,\z)$. We denote this expectation by
$$
  K(f):=\E_{(g,\z)\sim\rhozero}(K(f,g,\z)).
$$

\subsection{Average Analysis of \LV}

To talk about average complexity of \LV\ requires specifying a measure
for the set of inputs.
The most natural choice is the standard Gaussian distribution on $\Hd$.
Since $K(f)$ is invariant under scaling, we may equivalently
assume that $f$ is chosen in the unit sphere $S(\Hd)$ from the uniform distribution.
With this choice, we say a Las Vegas algorithm is
{\em average polynomial time} when the average
---over $f\in S(\Hd)$--- of its expected running time is polynomially
bounded in the size~$N$ of~$f$. The following result shows that \LV\ is
average polynomial time. It is essentially the main result in~\cite{BePa08b}
(modulo the existence of \ALH\ and with specific constants).

\begin{theorem}\label{thm:A}
The average of the expected number of iterations of Algorithm \LV\
is bounded as ($n \ge 4$)
$$
   \E_{f\in S(\Hd)} K(f) \leq 4185\, D^{3/2} N(n+1) .
$$
\end{theorem}

\subsection{Smoothed Analysis of \LV}\label{subsec:SA}

A smoothed analysis of an algorithm consists of bounding, for all
possible input data $\of$, the average of its running time (its
expected running time if it is a Las Vegas algorithm) over small
perturbations of $\of$. To perform such an analysis, a family of
measures (parameterized by a parameter $r$ controlling the size of
the perturbation) is considered with the
following characteristics:

(1)\quad the density of an element $f$ depends only on the distance
$\|f-\of\|$.

(2)\quad the value of $r$ is closely related to the variance of
$\|f-\of\|$.

\noindent Then, the average above is estimated as a function of the
data size $N$ and the parameter $r$, and a satisfying result, which
is described by the expression {\em smoothed polynomial time},
demands that this function is polynomially bounded in $r^{-1}$ and~$N$.
Possible choices for the measures' family are the Gaussians
$N(\of,\sigma^2\Id)$ (used, for instance,
in~\cite{CDW:05,sast:03,ST:04,Wsch:04}) and the uniform measure on
disks $B(\of,r)$ (used
in~\cite{AmelunxenBuergisser:08,BCL:06a,BCL:06c}).
Other families may also be used and an emerging impression is
that smoothed analysis is robust in the sense that its dependence
on the chosen family of measures is low.
This tenet was argued for in~\cite{CuHaLo:08} where
a uniform measure is replaced by an adversarial measure (one having
a pole at $\of$) without a significant loss in the estimated
averages.

In this paper, for reasons of technical simplicity and consistency
with the rest of the exposition, we will work with truncated
Gaussians defined as follows.
For $\of\in\Hd$ and $\sigma>0$
we shall denote by $N(\of,\sigma^2\Id)$
the Gaussian distribution on $\Hd \simeq \R^{2N}$
(defined with respect to the  Bombieri-Weyl basis)
with mean $\of$ and covariance matrix $\sigma^2\Id$.
Further, for $A>0$ let
$P_{A,\sigma}:=\Prob\{\|f\|\leq A\mid f\sim N(0,\sigma^2\Id)\}$.
We define the
{\em truncated Gaussian} $N_A(\of,\sigma^2\Id)$
with center $\of\in \Hd$ as the probability
measure on $\Hd$ with density
\begin{equation}\label{eq:densA}
   \rho(f)=\left\{
   \begin{array}{ll}
    \frac{\rho_{\of,\s}(f)}{P_{A,\sigma}}
    & \mbox{if $\|f-\of\|\leq A$}\\[3pt]
    0 & \mbox{otherwise,}
   \end{array}\right.
\end{equation}
where $\rho_{\of,\s}$ denotes the density of $N(\of,\sigma^2\Id)$.
Note that $N_A(\of,\sigma^2\Id)$
is isotropic around its mean~$\of$.

For our smoothed analysis we will take $A=\sqrt{2N}$. In this case,
we have $P_{A,\sigma}\geq\frac12$ for all $\sigma\leq 1$
(Lemma~\ref{lem:X}).
Note also that $\Var(\|f-\of\|)\leq\sigma^2$,
so that any upper bound polynomial in
$\sigma^{-2}$ is also an upper bound polynomial in $\Var(\|f-\of\|)^{-1}$.

We can now state our smoothed analysis result for \LV.

\begin{theorem}\label{th:main-1}
For any $0<\s\le 1$, Algorithm \LV\ satisfies
$$
 \sup_{\of\in S(\Hd)} \E_{f\sim N_A(\of,\s^2\Id)}K(f)
 \le 4185\, D^{3/2} \big(N+ 2^{-1/2}\sqrt{N}\big)(n+1)\,\frac{1}{\sigma} .
$$
\end{theorem}

\subsection{The Main Technical Result}\label{se:technical}

The technical heart of the proof of the mentioned results
on the average and smoothed analysis of \LV\
is the following smoothed analysis of the mean square condition number.

\begin{theorem}\label{th:mu2-bound}
Let $\oq\in\Hd$ and $\s>0$. For $q\in\Hd$ drawn from $N(\oq,\s^2\Id)$ we have
$$
   \E_{\Hd}\Big(\frac{\mu_2^2(q)}{\|q\|^2} \Big)\ \le\ \frac{e(n+1)}{2\s^2}.
$$
\end{theorem}

We note that no bound on the norm of $\oq$ is required here.
Indeed, using $\mu_2(\lambda q)=\mu_2(q)$,
it is easy to see that the assertion for $\oq,\s$ implies the
assertion for $\lambda \oq,\lambda\s$, for any $\lambda>0$.

\subsection{Condition-based Analysis of \LV}\label{subsec:CBA}

We are here interested in estimating~$K(f)$  for a fixed input system $f\in S(\Hd)$.
Such an estimate will have to depend on, besides $N$, $n$, and $D$, the
condition of $f$.
We measure the latter using Shub and Smale's~\cite{Bez1} $\mum(f)$
defined in~\eqref{eq:muSS}.
Our condition-based analysis of \LV\ is summarized in the
following statement.

\begin{theorem}\label{thm:CBA}
The expected number of iterations of Algorithm~\LV\ with
input $f\in S(\Hd)\setminus\Sigma$ is bounded as
$$
   K(f) \leq
   200411\, D^3 N(n+1)\mum^2(f) .
$$
\end{theorem}

\subsection{A Near Solution of Smale's 17th Problem}

We finally want to consider {\em deterministic} algorithms finding
zeros of polynomial systems. Our goal is to exhibit one such algorithm
working in nearly-polynomial average time, more precisely in
average time $N^{\Oh(\log\log N)}$. A first ingredient to do so is
a deterministic homotopy algorithm which is fast when $D$ is small.
This consists of algorithm \ALH\ plus the initial pair $(\oU,\bz_1)$, where
$\oU=(\oU_1,\ldots,\oU_n)\in S(\Hd)$ with
$\oU_i = \frac1{\sqrt{2n}}(X_0^{d_i} - X_i^{d_i})$
and $\bz_1=(1:1:\ldots:1)$.

We consider the following algorithm \FD\ (Moderate Degree):

\algo
\> Algorithm \FD\\[2pt]
\>{\bf input} $f\in \Hd$\\[2pt]
\>\> {\tt run \ALH\ on input $(f,\oU,\bz_1)$} \falgo
\medskip

We write $K_{\oU}(f):=K(f,\oU,\bz_1)$ for the number of iterations
of algorithm \FD\ with input $f$. We are interested in
computing the average over $f$ of $K_{\oU}(f)$ for $f$ randomly
chosen in $S(\Hd)$ from the uniform distribution.

The complexity of \FD\ is bounded as follows.

\begin{theorem}\label{thm:S17FD}
The average number of iterations of Algorithm \FD\ is bounded as
$$
   \E_{f\in S(\Hd)}K_{\oU}(f) \leq
    400821\, D^{3}\, N(n+1)^{D+1}.
$$
\end{theorem}

Algorithm \FD\ is efficient when $D$ is small, say, when $D\leq n$.
For $D>n$ we use another approach, namely, a real number algorithm designed
by Jim Renegar~\cite{rene:89} which in this case has a performance
similar to that of \FD\ when $D\leq n$. Putting both pieces together we
will reach our last main result.

\begin{theorem}\label{th:near17}
There is a {\em deterministic} real number algorithm that on input~$f\in\Hd$
computes an approximate zero of~$f$ in average time
$N^{\Oh(\log\log N)}$, where $N=\dim\Hd$ measures the size of the input~$f$.
Moreover, if we restrict data to polynomials satisfying
$$
    D\le n^{\frac1{1+\e}} \quad\mbox{ or }\quad D\ge n^{1+\e},
$$
for some fixed $\e>0$,
then the average time of the algorithm is polynomial
in the input size~$N$.
\end{theorem}

\section{Complexity Analysis of \ALH}\label{sec:ALH}

The goal of this section is to prove
Theorem~\ref{thm:main_path_following}.
An essential component in this proof is an
estimate of how much $\mun(f,\z)$ changes when $f$ or $\z$ (or
both) are slightly perturbed. The following result gives upper and
lower bounds on this variation. It is a precise version, with
explicit constants, of Theorem~1 of~\cite{Bez6}.

\begin{proposition}\label{prop:apps}
Assume $D\geq 2$. Let $0<\e\leq 0.13$ be arbitrary and
$C\leq \frac{\e}{5.2}$.
For all $f,g\in S(\Hd)$
and all $x,\zeta\in \proj^n$, if
$d(f,g) \leq \frac{C}{D^{1/2}\mun(f,\zeta)}$ and $d(\zeta,x) \leq
 \frac{C}{D^{3/2}\mun(f,\zeta)}$,
then
\begin{center}
$\displaystyle
   \frac{1}{1+\e}\, \mun(g,x)\leq \mun(f,\zeta) \leq
     (1+\e)\mun(g,x).
$
\end{center}
\vspace*{-0.7cm}{{\mbox{}\hfill\qed}}\vspace{0.6cm}
\end{proposition}

In what follows, we will fix the constants as
$\e=0.13$ and $C=\frac{\e}{5.2} = 0.025$.

\begin{remark}\label{rem:Ceps}
The constants $C$ and $\e$ implicitly occur in the statement of
Theorem~\ref{thm:main_path_following} since the 245 therein is a
function of these numbers. But their role is not limited to this
since they also occur in the algorithm \ALH\
in the parameter $\lambda=\frac{C(1-\e)}{2(1+\e)^4}$
controlling the update $\tau+\Delta\tau$ of $\tau$.
We note that for the
former we could do without precise values by using the big Oh
notation. In contrast, we cannot talk of a constructive procedure
unless all of its steps are precisely given.
\end{remark}

\begin{proof}[Proof of Theorem~\ref{thm:main_path_following}]
Let $0=\tau_0<\tau_1<\ldots<\tau_k=1$ and 
$\z_0=x_0,x_1,\ldots,x_k$
be the sequences of $\tau$-values and points in $\proj^n$
generated by the algorithm \ALH.
To simplify notation we write $q_i$ instead of $q_{\tau_i}$ and
$\z_i$ instead of $\zeta_{\tau_i}$.
\smallskip

We claim that, for $i=0,\ldots,k-1$, the following inequalities
are true:
\medskip

\noindent {\bf (a)}\quad $\displaystyle
 \dpr(x_i,\zeta_i)\leq \frac{C}{D^{3/2}\mun(q_i,\zeta_i)}$
\smallskip

\noindent {\bf (b)}\quad $\displaystyle
    \frac{\mun(q_i,x_i)}{(1+\e)}\leq
    \mun(q_i,\zeta_i)\leq (1+\e)\mun(q_i,x_i)$
\smallskip

\noindent {\bf (c)}\quad $\displaystyle
 \dsp(q_i,q_{i+1})\leq \frac{C}{D^{3/2}\mun(q_i,\zeta_i)}$
\smallskip

\noindent {\bf (d)}\quad $\displaystyle
 \dpr(\zeta_i,\zeta_{i+1})\leq \frac{C}{D^{3/2}\mun(q_i,\zeta_i)}
 \,\frac{(1-\e)}{(1+\e)}$
\smallskip

\noindent {\bf (e)}\quad $\displaystyle
 \dpr(x_i,\zeta_{i+1})\leq \frac{2C}{(1+\e)D^{3/2}\mun(q_i,\zeta_i)}$
\medskip

We proceed by induction showing that
$$
  ({\bf a},i)\Rightarrow ({\bf b},i)\Rightarrow
  \big(({\bf c},i) \mbox{\rm\ and } ({\bf d},i)\big)\Rightarrow
  ({\bf e},i)\Rightarrow ({\bf a},i+1).
$$
Inequality (a) for $i=0$ is trivial.

Assume now that (a) holds for some $i \le k-1$. Then,
Proposition~\ref{prop:apps} (with $f=g=q_i$) implies
$$
   \frac{\mun(q_i,x_i)}{(1+\e)}\leq
   \mun(q_i,\zeta_i)\leq (1+\e)\mun(q_i,x_i)
$$
and thus (b). We now show (c) and (d). 
To do so, put $p_\tau := \frac{q_\tau}{\|q_\tau\|}$ and let $\tau_*>\tau_i$ be such that
$\int_{\tau_i}^{\tau_*}(\|\dot{p}_\tau\|+\|\dot{\zeta}_\tau\|)d\tau
=\frac{C}{D^{3/2}\mun(q_i,\zeta_i)}\,\frac{(1-\e)}{(1+\e)}$ or
$\tau_*=1$, whichever the smallest.
Then, for all $t\in[\tau_i,\tau_*]$,
\begin{eqnarray*}
 \dpr(\zeta_i,\zeta_t) &=& \int_{\tau_i}^t\|\dot{\zeta}_\tau\|\, d\tau
     \leq \int_{\tau_i}^{\tau_*}(\|\dot{p}_\tau\|+\|\dot{\zeta}_\tau\|)d\tau \\
    &\leq& \frac{C}{D^{3/2}\mun(q_i,\zeta_i)}\,\frac{(1-\e)}{(1+\e)}
\end{eqnarray*}
and, similarly,
$$
 \dsp(q_i,q_t)
  \leq \frac{C}{D^{3/2}\mun(q_i,\zeta_i)}\,\frac{(1-\e)}{(1+\e)}
  \leq \frac{C}{D^{3/2}\mun(q_i,\zeta_i)}.
$$
It is therefore enough to show that $\tau_{i+1}\leq \tau_*$. This is
trivial if $\tau_*=1$. We therefore assume $\tau_*<1$. The two
bounds above allow us to apply Proposition~\ref{prop:apps} and
to deduce, for all $\tau\in[\tau_i,\tau_*]$,
$$
    \mun(q_\tau,\zeta_\tau) \leq (1+\e) \mun(q_i,\zeta_i).
$$
From $\|\dot{\z}_\tau\| \le \mun(q_\tau,\z_\tau)\, \|\dot{p}_\tau\|$
(cf.~\cite[\S12.3-12.4]{bcss:95})
and $\mun(q_\tau,\z_\tau)\ge 1$
it follows that
\begin{equation*}
\begin{split}
  \frac{C}{D^{3/2}\mun(q_i,\zeta_i)}\,\frac{(1-\e)}{(1+\e)}
  =  \int_{\tau_i}^{\tau_*}(\|\dot{p}_\tau\|
       +\|\dot{\zeta}_\tau\|)d\tau
  \leq \int_{\tau_i}^{\tau_*}2\mun(q_\tau,\zeta_\tau)
    \|\dot{p}_\tau\| d\tau \\
  \leq\  2(1+\e)\mun(q_i,\zeta_i) \int_{\tau_i}^{\tau_*}
   \|\dot{p}_\tau\| d\tau
  \ \leq\  2\dsp(q_i,q_{\tau_*})(1+\e)\mun(q_i,\zeta_i) .
\end{split}
\end{equation*}
Consequently, using (b), we obtain
$$
  \dsp(q_i,q_{\tau_*}) \geq \frac{C(1-\e)}{2(1+\e)^2D^{3/2}\mun^2(q_i,\zeta_i)}
  \geq \frac{C(1-\e)}{2(1+\e)^4 D^{3/2}\mun^2(q_i,x_i)} .
$$
The parameter $\lambda$ in \ALH\ is chosen as $\frac{C(1-\e)}{2(1+\e)^4}$ (or slightly less).
By the definition of $\tau_{i+1}-\tau_i$ in \ALH\ we have
$\a(\tau_{i+1}-\tau_i) = \frac{\lambda}{D^{3/2}\mun^2(q_i,x_i)}$.
So we obtain
$$
 \dsp(q_i,q_{\tau_*}) \geq \a(\tau_{i+1}-\tau_i)\,=\,\dsp(q_i,q_{i+1}).
$$
This implies $\tau_{i+1}\leq \tau_*$ as claimed and hence,
inequalities (c) and (d). With them, we may apply
Proposition~\ref{prop:apps} to deduce, for all
$\tau\in[\tau_i,\tau_{i+1}]$,
\begin{equation}\label{eq:b1}
    \frac{\mun(q_i,\zeta_i)}{1+\e} \leq \mun(q_\tau,\zeta_\tau)
    \leq (1+\e) \mun(q_i,\zeta_i).
\end{equation}
Next we use the triangle inequality, (a), and (d), to obtain
\begin{eqnarray*}
  \dpr(x_i,\zeta_{i+1}) &\leq& \dpr(x_i,\zeta_i) +\dpr(\zeta_i,\zeta_{i+1}) \\
 &\leq& \frac{C}{D^{3/2}\mun(q_i,\zeta_i)}
    +\frac{C}{D^{3/2}\mun(q_i,\zeta_i)}\frac{(1-\e)}{(1+\e)} \\
  &=& \frac{2C}{(1+\e)D^{3/2}\mun(q_i,\zeta_i)},
\end{eqnarray*}
which proves (e).
Theorem~\ref{thm:gamma} yields that $x_i$ is an approximate zero
of~$q_{i+1}$ associated with its zero $\z_{i+1}$.
Indeed, by our choice of $C$ and $\e$, we have
$2C\le u_0(1+\e)$ and hence
$\dpr(x_i,\zeta_{i+1})\leq\frac{u_0}{D^{3/2}\mun(q_i,\zeta_i)}$.
Therefore,
$x_{i+1}=N_{q_{i+1}}(x_i)$ satisfies
$$
\dpr(x_{i+1},\zeta_{i+1})\leq \frac12\,\dpr(x_i,\z_{i+1}).
$$
Using (e) and
the right-hand inequality in~\eqref{eq:b1} with $t=t_{i+1}$, we
obtain
$$
   \dpr(x_{i+1},\zeta_{i+1})
  \leq \frac{C}{(1+\e)D^{3/2}\mun(q_i,\zeta_i)}\\
  \leq \frac{C}{D^{3/2}\mun(q_{i+1},\zeta_{i+1})} ,
$$
which proves (a) for $i+1$. The claim is thus proved.

The estimate
$ \dpr(x_{k},\zeta_{k}) \le \frac{C}{D^{3/2}\mun(q_k,\zeta_k)}$
just shown for $i=k-1$
implies by Theorem~\ref{thm:gamma} that the
returned point~$x_k$ is an approximate zero of
$q_k=f$ with associated zero~$\z_1$.

Consider now any $i\in\{0,\ldots,k-1\}$. Using~\eqref{eq:b1} and (b)
we obtain
\begin{eqnarray*}
  \int_{\tau_i}^{\tau_{i+1}} \mun^2(q_\tau,\zeta_\tau)d\tau
  &\geq&
  \int_{\tau_i}^{\tau_{i+1}
  }\frac{\mun^2(q_i,\zeta_i)}{(1+\e)^2}d\tau
  =  \frac{\mun^2(q_i,\zeta_i)}{(1+\e)^2} (\tau_{i+1}-\tau_i) \\
  &\geq&\frac{\mun^2(q_i,x_i)}{(1+\e)^4}(\tau_{i+1}-\tau_i)\\
  &=& \frac{\mun^2(q_i,x_i)}{(1+\e)^4}
       \frac{\lambda}{\a D^{3/2}\mun^2(q_i,x_i)}\\
  &=&  \frac{\lambda}{(1+\e)^4\a D^{3/2}} \ \ge\  \frac1{245} \frac1{\a D^{3/2}} .
\end{eqnarray*}
This implies
$$
  \int_0^1 \mun^2(q_\tau,\zeta_\tau)d\tau
   \ge \frac{k}{245} \frac1{\a D^{3/2}} ,
$$
which proves the stated bound on~$k$.
\end{proof}
\bigskip

\section{A Useful Change of Variables}

We first draw a conclusion of Theorem~\ref{thm:main_path_following},
that we will need several times.
Recall the definition~\eqref{eq:mu2} of the mean square condition number~$\mu_2(q)$.

\begin{proposition}\label{cor:main_path_following}
The expected number of iterations of \ALH\ on input~$f\in\Hd\setminus\Sigma$ is bounded as
$$
    K(f) \le 245\, D^{3/2}\E_{g\in S(\Hd)}
    \left(\dsp(f,g)\int_0^1\mu_2^2(q_\tau)d\tau \right).
$$
\end{proposition}

\begin{proof}
Fix $g\in\Hd$ such that the segment $E_{f,g}$ does not intersect the
discriminant variety~$\Sigma$
(which is the case for almost all~$g$, as $f\not\in\Sigma$).
To each of the zeros~$\z^{(i)}$ of $g$ there corresponds a lifting
$[0,1]\to V,\tau\mapsto (q_\tau,\z_\tau^{(i)})$
of $E_{f,g}$
such that $\z_0^{(i)}=\z^{(i)}$.
Theorem~\ref{thm:main_path_following} states that
$$
      K(f,g,\z^{(i)}) \le 245\,D^{3/2}\,\dsp(f,g)\int_0^1\mun^2(q_\tau,\zeta_\tau^{(i)})\,d\tau.
$$
Since $\z_\tau^{(1)},\ldots,\z_\tau^{(\Dn)}$ are the zeros of $q_\tau$,
we have by the definition~\eqref{eq:mu2} of the mean square condition number
\begin{equation}\label{eq:cor3.4}
\frac1{\Dn} \sum_{i=1}^\Dn K(f,g,\z^{(i)}) \le 245\,D^{3/2}\,\dsp(f,g)\int_0^1\mu_2^2(q_\tau)\,d\tau.
\end{equation}
The assertion follows now from (compare the forthcoming Lemma~\ref{le:rho1})
\begin{equation*}
  K(f)=\E_{(g,\z)\sim\rhozero}(K(f,g,\z))
 = \E_{g\in S(\Hd)}\left(\frac1{\Dn}\sum_{i=1}^\Dn K(f,g,\z^{(i)}) \right) . \qedhere
\end{equation*}
\end{proof}

The remaining of this article is devoted to prove
Theorems~\ref{thm:A}--\ref{th:near17}.
All of them involve expectations
---over random $f$ and/or $g$--- of the integral
$
    \int_0^1\mu_2^2(q_{\tau})d\tau .
$
In all cases, we will eventually deal with such an expectation with $f$ and
$g$ Gaussian. Since a linear combination (with fixed coefficients)
of two such Gaussian systems is Gaussian as well, it is convenient
to parameterize the interval $E_{f,g}$ by a parameter $t\in[0,1]$
representing a ratio of Euclidean distances (instead of a ratio of angles as $\tau$ does).
Thus we write, abusing notation, $q_t=tf+(1-t)g$.
For fixed $t$, as noted before, $q_t$ follows a Gaussian law.
For this new parametrization we have the following result.

\begin{proposition}\label{prop:a-ALH}
Let $f,g\in\Hd$ be $\R$-linearly independent and $\tau_0\in[0,1]$. Then
\begin{equation*}
 \dsp(f,g)\int_{\tau_0}^1 \mu_2^2 (q_\tau)d\tau
 \leq \int_{t_0}^1 \|f\|\,\|g\|\,\frac{\mu_2^2(q_t)}{\|q_t\|^2}\, dt,
\end{equation*}
where
$$
   t_0=\frac{\|g\|}
     {\|g\| + \|f\| (\sin\a\cot(\tau_0\a)-\cos\a)}
$$
is the fraction of the Euclidean distance $\|f-g\|$ corresponding to
the fraction $\tau_0$ of the angle $\a=\dsp(f,g)$.
\end{proposition}

\begin{proof}
For $t\in [0,1]$, abusing notation, we let $q_t=tf+(1-t)g$ and $\tau(t)\in[0,1]$ be such
that $\tau(t)\a$ is the angle between $g$ and $q_t$. This
defines a bijective map $[t_0,1]\to [\tau_0,1], t\mapsto \tau(t)$. We denote its
inverse by $\tau\mapsto t(\tau)$.
We claim that
\begin{equation}\label{speed}
 \frac{d\tau}{dt} = \frac{\sin\a}{\a}\, \frac{\|f\|\cdot \|g\|}{\|q_t\|^2} .
\end{equation}
Note that the stated inequality easily follows from this claim by the transformation
formula for integrals together with the bound $\sin\a\le 1$.

To prove Claim~\eqref{speed}, denote
$r=\|f\|$ and $s=\|g\|$.
We will explicitly compute $t(\tau)$ by some elementary geometry.
For this, we introduce cartesian coordinates in the
plane spanned by $f$ and $g$ and
assume that $g$ has the coordinates $(s,0)$
and $f$ has the coordinates $(r\cos\a,r\sin\a)$,
see Figure~\ref{fig:1}.
\begin{figure}[H]
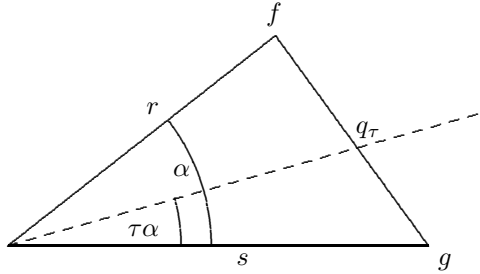

\begin{center}
   \input situation.pictex
\end{center}
\caption{Computing $t(\tau)$.}\label{fig:1}
\end{figure}

Then, the lines determining $q_\tau$ have the equations
$$
    x=y\,\frac{\cos (\tau\a)}{\sin(\tau\a)}
   \qquad\mbox{ and }\qquad
   x=y\,\frac{r\cos \a-s}{r\sin \a}+s
$$
from where it follows that the coordinate $y$ of $q_\tau$ is
\begin{equation}\label{eq:y}
  y=\frac{rs\sin\a\sin(\tau\a)}
     {r\sin\a\cos(\tau\a)-r\cos\a\sin(\tau\a)+s\sin(\tau\a)}.
\end{equation}
Since $t(\tau)=\frac{y}{r\sin\a}$ it follows that
\begin{equation}\label{eq:lambda}
    t(\tau) =\frac{s}{r\sin\a\cot(\tau\a)-r\cos\a+s}.
\end{equation}
This implies the stated formula for $t_0=t(\tau_0)$.
Differentiating with respect to $\tau$,
using~\eqref{eq:y} and  $\sin(\tau\a)=\frac{y}{\|q_\tau\|}$,
we obtain from \eqref{eq:lambda}
\begin{eqnarray*}
  \frac{dt}{d\tau}&=&
   \frac{\a rs\sin\a}
     {(r\sin\a\cos(\tau\a)-r\cos\a\sin(\tau\a)+s\sin(\tau\a))^2}\\
   &=& \frac{\a y^2}{rs \sin^2(\tau\a)\sin\a}
    = \frac{\a \|q_{t(\tau)}\|^2}{rs\sin\a}.
\end{eqnarray*}
This finishes the proof of Claim~\eqref{speed}.
\end{proof}

In all the cases we will deal with, the factor $\|f\|\,\|g\|$ will
be easily bounded and factored out the expectation. We will
ultimately face the problem of estimating expectations of
$\frac{\mu_2^2(q_t)}{\|q_t\|^2}$ for different choices of $\oq_t$ and $\sigma_t$.
This is achieved by Theorem~\ref{th:mu2-bound} stated in \S\ref{se:technical}.





\section{Analysis of \LV\ }\label{se:pfLV}

We derive here from Theorem~\ref{th:mu2-bound} our main results on the
average and smoothed analysis of \LV\ stated in \S\ref{sec:main_results}.
The proof of Theorem~\ref{th:mu2-bound} is postponed to
Sections~\ref{se:smMC}--~\ref{se:red}.

\subsection{Average-case Analysis of \LV\ (proof)}\label{se:Aproof}

To warm up, we first prove Theorem~\ref{thm:A}, which
illustrates the blending of the previous results in a simpler setting.

In the following we set $A:=\sqrt{2N}$ and write
$P_{A,\s} = \Prob\{\|f\| \leq A \mid f \sim N(0,\s^2\Id)\}$
for $\s>0$.

\begin{lemma}\label{lem:X}
We have $P_{A,\s} \ge \frac12$
for all $0<\sigma\leq 1$.
\end{lemma}

\begin{proof}
Clearly it suffices to assume $\s=1$.
The random variable $\|f\|^2$ is chi-square distributed with
$2N$ degrees of freedom. Its mean equals $2N$.
In~\cite[Corollary~6]{choi:94} it is shown that
the median of a chi-square distribution is always
less than its mean.
\end{proof}

\begin{proof}[Proof of Theorem~\ref{thm:A}]
We use Proposition~\ref{cor:main_path_following} to obtain
\begin{eqnarray*}
  \E_{f\in S(\Hd)}K(f)
 &\le & 245\, D^{3/2}\, \E_{f\in S(\Hd)} \E_{g\in S(\Hd)}
 \bigg(\dsp(f,g) \int_0^1 \mu_2^2(q_{\tau})d\tau\bigg) \\
 &=& 245\, D^{3/2} \E_{f\sim N_A(0,\Id)}
  \E_{g\sim N_A(0,\Id)}\bigg(\dsp(f,g) \int_0^1 \mu_2^2(q_{\tau})d\tau\bigg).
\end{eqnarray*}
The equality follows from the fact that, since both
$\dsp(f,g)$ and $\mu_2^2(q_{\tau})$ are homogeneous of
degree 0 in both $f$ and $g$, we may replace the uniform
distribution on $S(\Hd)$ by any rotationally invariant distribution on $\Hd$,
in particular by the centered truncated Gaussian $N_A(0,\Id)$ defined in \eqref{eq:densA}.
Now we use Proposition~\ref{prop:a-ALH} (with $\tau_0=0$) to get
\begin{equation}\label{eq:auflosen}
 \E_{f\in S(\Hd)} K(f) \;\leq\;
   245 D^{3/2} A^2 \E_{f\sim N_A(0,\Id)} \E_{g\sim N_A(0,\Id)}
   \bigg(\int_0^1 \frac{\mu_2^2(q_t)}{\|q_t\|^2}\, dt \bigg).
\end{equation}
Denoting by $\rho_{0,1}$ the density of $N(0,\Id)$,
the right-hand side of \eqref{eq:auflosen} equals
\begin{align*}
&
  245\, D^{3/2} \frac{A^2}{P_{A,1}^2}
   \int_{\|f\|\leq A} \int_{\|g\|\leq A}
   \bigg(\int_0^1 \frac{\mu_2^2(q_t)}{\|q_t\|^2}\, dt\bigg)
     \;\rho_{0,1}(g)\,\rho_{0,1}(f)\, dg\, df\\
 &\leq \;
   245\, D^{3/2} \frac{A^2}{P_{A,1}^2}
   \E_{f\sim N(0,\Id)} \E_{g\sim N(0,\Id)}
   \bigg(\int_0^1 \frac{\mu_2^2(q_t)}{\|q_t\|^2}\, d t \bigg) \\
 &=\;
   245\, D^{3/2} \frac{A^2}{P_{A,1}^2}
   \int_0^1 \E_{q_t\sim N(0,(t^2+(1-t)^2)\Id)}
   \bigg(\frac{\mu_2^2(q_t)}{\|q_t\|^2}\bigg)\, d t,
\end{align*}
where the last equality follows
from the fact that,
for fixed $t$, the random polynomial system $q_t=tf + (1-t)g$
has a Gaussian distribution with law $N(0,\s_t^2\Id)$,
where $\s_t^2:=t^2+(1-t)^2$.
Note that we deal with nonnegative integrands,
so the interchange of integrals is justified
by Tonelli's theorem.
By Lemma~\ref{lem:X} we have
$\frac{A^2}{P_{A,1}^2} \le 8N$.

We now apply Theorem~\ref{th:mu2-bound} to deduce that
$$
 \int_0^1 \E_{q_t\sim N(0,\s_t^2\Id)} \left(\frac{\mu_2^2(q_t)}{\|q_t\|^2}\right)\, dt
  \leq \frac{e(n+1)}{2}\, \int_0^1 \frac{dt}{t^2+(1-t)^2}
  = \frac{e\pi(n+1)}{4}.
$$
Consequently,
\begin{equation*}
   \E_{f\in S(\Hd)} K(f)
   \le 245\, D^{3/2}\cdot 8N\cdot \frac{e\pi(n+1)}{4}
   \le 4185 \, D^{3/2} N(n + 1) .  \qedhere
\end{equation*}
\end{proof}

\begin{remark}
The proof (modulo the existence of \ALH) for the average complexity of \LV\
given by  Beltr\'an and Pardo in~\cite{BePa08b} differs from the
one above. It relies on the fact (elegantly shown by using integral geometry
arguments) that, for all $\tau\in[0,1]$, when $f$ and $g$ are uniformly drawn
from the sphere, so is $q_\tau/\|q_\tau\|$.
The extension of this argument to more general situations appears to be considerably more involved.
In contrast, as we shall shortly see, the argument based on Gaussians
in the proof above carries over, {\em mutatis mutandis}, to the smoothed analysis context.
\end{remark}

\subsection{Smoothed Analysis of \LV\ (proof)}\label{sec:SA}

The smoothed analysis of \LV\ is shown similarly to its average-case analysis.

\begin{proof}[Proof of Theorem~\ref{th:main-1}]
Fix $\of\in S(\Hd)$.
Reasoning as in the proof of Theorem~\ref{thm:A} and
using $\|f\|\leq\|\of\|+\|f-\of\|\leq1+A$, we show that
\begin{equation*}
\begin{split}
  \E_{f\sim N_A(\of,\sigma^2\Id)} K(f)\
  &\leq\
  245 D^{3/2} \frac{(A+1)A}{P_{A,\sigma}P_{A,1}}
   \E_{f\sim N(\of,\s^2\Id)} \E_{g\sim N(0,\Id) }
     \bigg( \int_0^1 \frac{\mu_2^2(q_t)}{\|q_t\|}\, dt \bigg) \\
  &=\
  245 D^{3/2} \frac{(A+1)A}{P_{A,\sigma}P_{A,1}}
   \int_0^1 \E_{q_t\sim N(\oq_t,\sigma_t^2\Id)}
   \bigg(\frac{\mu_2^2(q_t)}{\|q_t\|}\bigg)\; dt
\end{split}
\end{equation*}
with $\oq_t=t\of$ and $\sigma_t^2=(1-t)^2+\sigma^2t^2$.
We now apply Theorem~\ref{th:mu2-bound} to deduce
$$
 \int_0^1\,\E_{q_t\sim N(\oq_t,\sigma_t^2\Id)} \left(\frac{\mu_2^2(q_t)}{\|q_t\|^2}\right)\, dt
 \ \leq\ \frac{e(n+1)}{2} \int_0^1 \frac{dt}{(1-t)^2 +\s^2t^2}
 = \frac{e\pi(n+1)}{4\s}.
$$
Consequently, using Lemma~\ref{lem:X}, we get
$$
   \E_{f\sim N_A(\of,\sigma^2\Id)} K(f) \leq
 245\, D^{3/2}\cdot 4\cdot (2N+\sqrt{2N})\,\frac{e\pi(n+1)}{4\s}
$$
which proves the assertion.
\end{proof}

The next two sections are devoted to the proof of Theorem~\ref{th:mu2-bound}.
First, in Section~\ref{se:smMC}, we give a particular
smoothed analysis of a matrix condition number (Proposition~\ref{pro:fiberE}).
Then, in Section~\ref{se:red}, we reduce Theorem~\ref{th:mu2-bound}
to this smoothed analysis of matrix condition numbers.

\section{Smoothed Analysis of a Matrix Condition Number}\label{se:smMC}

In the following we fix $\oA\in\C^{n\times n}$, $\s>0$ and
denote by $\rho$ the Gaussian density of $N(\oA,\s^2\Id)$
on $\C^{n\times n}$.
Moreover, we consider the related density
\begin{equation}\label{eq:rhotilde}
 \tilde{\rho}(A) = c^{-1}\, |\det A|^2\, \rho(A)
 \quad\mbox{where } c := \E_{A\sim \rho}(|\det A|^2) .
\end{equation}

The following result is akin to a smoothed analysis
of the matrix condition number
$\kappa(A)=\|A\|\cdot\|A^{-1}\|$,
with respect to the probability densities $\tilde{\rho}$
that are not Gaussian, but closely related to Gaussians.

\begin{proposition}\label{pro:fiberE}
We have
$\E_{A\sim\tilde{\rho}}\big(\|A^{-1}\|^2\big)\ \le\ \frac{e(n+1)}{2\s^2}$.
\end{proposition}

The proof is based on ideas in Sankar et al.~\cite[\S3]{sast:03},
see also~\cite{BC:10}.
We will actually prove tail bounds from which the stated
bound on the expectation easily follows.


We denote by
$\bS^{n-1}:=\{ \z\in\C^{n} \mid \|\z\| = 1\}$
the unit sphere in $\C^n$.

\begin{lemma}\label{le:tbrho}
For any $v\in \bS^{n-1}$ and any $t>0$ we have
$$
 \Prob_{A\sim\tilde{\rho}} \Big\{ \|A^{-1}v\| \ge t \Big\}
 \ \le\ \frac1{4\s^4 t^4}.
$$
\end{lemma}

\begin{proof}
We first claim that, because of unitary invariance,
we may assume that
$v=e_n:=(0,\ldots,0,1)$.
To see this, take $S\in U(n)$ such that $v=Se_n$.
Consider the isometric map
$A\mapsto B=S^{-1}A$ which transforms the density
$\tilde{\rho}(A)$ to a density of the same form, namely
$$
 \tilde{\rho}'(B) = \tilde{\rho}(A)
 = c^{-1} |\det A|^2 \rho(A) = c^{-1} |\det B|^2 \rho'(B),
$$
where $\rho'(B)$ denotes the density of
$N(S^{-1}\oA,\s^2\Id)$ and
$c=\E_{\rho}(|\det A|^2) = \E_{\rho'}(|\det B|^2)$.
Thus the assertion for $e_n$ and random $B$
(chosen from any isotropic Gaussian distribution)
implies the assertion for $v$ and $A$, noting that
$A^{-1}v=B^{-1}e_n$. This proves the claim.

Let $a_i$ denote the $i$th row of $A$. Almost surely,
the rows $a_1,\ldots,a_{n-1}$ are linearly independent.
We are going to characterize $\|A^{-1}e_n\|$
in a geometric way. Let
$S_n:=\spann\{a_1,\ldots,a_{n-1}\}$ and denote by $a_n^\perp$
the orthogonal projection of $a_n$ onto $S_n^\perp$.
Consider $w:=A^{-1}e_n$, which is the $n$th column
of $A^{-1}$. Since $A A^{-1}=\Id$ we have
$\langle w,a_i\rangle =0$ for $i=1,\ldots,n-1$
and hence $w\in S_n^\perp$.
Moreover, $\langle w,a_n\rangle =1$, so
$\|w\|\, \|a_n^\perp\| =1$ and we arrive at
\begin{equation}\label{eq:star}
 \|A^{-1}e_n\| = \frac1{\|a_n^\perp\|}.
\end{equation}

Let $A_n\in\C^{(n-1)\times n}$ denote the matrix
obtained from~$A$ by omitting $a_n$.
We shall write $\vol(A_n)= \det(AA^*)^{1/2}$
for the $(n-1)$-dimensional volume of the parallelepiped spanned
by the rows of $A_n$.
Similarly, $|\det A|$ can be interpreted as the $n$-dimensional
volume of the parallelepiped spanned by the rows of $A$.

Now we write $\rho(A) =\rho_1(A_n)\rho_2(a_n)$ where
$\rho_1$ and $\rho_2$ are the density functions of
$N(\bar{A}_{n},\sigma^2\Id)$ and
$N(\bar{a}_n,\sigma^2\Id)$, respectively
(the meaning of $\bar{A}_n$ and $\bar{a}_n$
being clear). Moreover, note that
$$
    \vol(A)^2=\vol(A_{n})^2 \, \|a_n^\perp\|^2 .
$$
Fubini's Theorem combined with \eqref{eq:star} yields for $t>0$
\begin{eqnarray}\notag
     \int_{\|A^{-1}e_n\| \ge t} \vol(A)^2\rho(A)\, dA &=&
     \int_{A_{n}\in\C^{(n-1)\times n}} \vol(A_{n})^2\,
     \rho_1(A_{n})\nonumber \\ \label{eq:WX}
  & &\qquad \cdot  \left(\int_{\|a_n^\perp\|\leq 1/t}
     \|a_n^\perp\|^2\rho_2(a_n)\, da_n\right) dA_{n}.
\end{eqnarray}
We next show that for fixed, linearly independent
$a_1,\ldots,a_{n-1}$ and $\lambda>0$
\begin{equation}\label{eq:C4}
    \int_{\|a_n^\perp\|\leq\lambda}
    \|a_n^\perp\|^2\rho_2(a_n)\, da_n
    \leq \frac{\lambda^4}{2\sigma^2}.
\end{equation}

For this, note that
$a_n^\perp\sim N(\bar{a}_n^\perp,\sigma^2\Id)$
in $S_n^\perp\simeq\C$ where $\bar{a}_n^\perp$
is the orthogonal
projection of $\bar{a}_n$ onto $S_n^\perp$. Thus,
proving~\eqref{eq:C4} amounts to showing
$$
    \int_{|z|\leq\lambda}|z|^2\rho_{\bar{z}}(z)dz
    \leq \frac{\lambda^4}{2\sigma^2}
$$
for the Gaussian density
$\rho_{\bar{z}}(z)=\frac{1}{2\pi\sigma^2}
e^{-\frac{1}{2\sigma^2}|z-\bar{z}|^2}$ of $z\in\C$,
where $\bar{z}\in\C$.
Clearly, it is enough to show that
\begin{equation*}\label{eq:enough}
     \int_{|z|\leq\lambda}\rho_{\bar{z}}(z)dz
     \leq \frac{\lambda^2}{2\sigma^2}.
\end{equation*}
Without loss of generality we may assume that $\bar{z}=0$,
since the integral in the left-hand side is maximized at this value of~$\bar{z}$.
The substitution $z=\sigma w$
yields $dz=\sigma^2 dw$
($dz$ denotes the Lebesgue measure on $\R^2$)
and we get
\begin{eqnarray*}
       \int_{|z|\leq\lambda}\rho_{0}(z)dz
  &=&
      \int_{|w|\leq\frac{\lambda}{\sigma}} \frac{1}{2\pi}
     e^{-\frac{1}{2}|w|^2}\, dw
    = \int_0^{\frac{\lambda}{\sigma}} \frac{1}{2\pi}
     e^{-\frac{1}{2}r^2}2\pi r \, dr\\
  &=& -e^{-\frac{1}{2}r^2}\bigg|_0^{\frac{\lambda}{\sigma}}
    = 1-e^{-\frac{\lambda^2}{2\sigma^2}}
    \leq \frac{\lambda^2}{2\sigma^2},
\end{eqnarray*}
which proves inequality~\eqref{eq:C4}.

A similar argument shows that
\begin{equation}\label{eq:C4p}
    2\s^2 \ \le\
    \int_{}\ |z|^2\rho_{\bar{z}}(z)dz =
    \int_{}\ \|a_n^\perp\|^2\rho_2(a_n)\, da_n .
\end{equation}
Plugging in this inequality into~\eqref{eq:WX}
(with $t=0$) we conclude that
\begin{equation}\label{eq:volvol}
  2\s^2\,\E_{\rho_1}\big(\vol(A_n)^2\big)
  \ \le\  \E_{\rho}\big(\vol(A)^2\big).
\end{equation}
On the other hand,
plugging in~\eqref{eq:C4} with $\lambda=\frac1t$
into \eqref{eq:WX}, we obtain
$$
   \int_{\|A^{-1}e_n\|\ge t} \vol(A)^2 \rho(A)\, dA \ \le\
   \frac{1}{2\s^2t^4}\ \E_{\rho_1}\big(\vol(A_n)^2\big).
$$
Combined with~\eqref{eq:volvol} this yields
$$
  \int_{\|A^{-1}e_n\|\ge t} \vol(A)^2 \rho(A)\, dA \ \le\
  \frac1{4\s^4t^4}\, \E_{\rho}\big(\vol(A)^2\big).
$$
By the definition of the density $\tilde{\rho}$, this means that
$$
 \Prob_{A\sim\tilde{\rho}}\big\{ \|A^{-1}e_n\|\ge t\} \ \le\
  \frac1{4\s^4t^4} ,
$$
which was to be shown.
\end{proof}

\begin{lemma}\label{le:tvb}
For fixed $u\in \bS^{n-1}$, $0\le s\le 1$, and random~$v$
uniformly chosen in~$\bS^{n-1}$ we have
$$
 \Prob_{v} \Big\{ |u\transp  v| \ge s \Big\} = (1-s^2)^{n-1} .
$$
\end{lemma}

\begin{proof}
Recall the Riemannian distance $\dpr$ in
$\proj^{n-1}:=\proj(\C^n)$ from \eqref{eq:projdist}.
Accordingly, for $0\le\theta\le\pi/2$, we have
$$
 \Prob_{v} \Big\{ |u\transp  v| \ge \cos\theta \Big\}
 = \frac{\vol\big\{[v]\in\proj^{n-1}\mid
   \dpr([u],[v])\le\theta \big\}}{\vol\,\proj^{n-1}}
 = (\sin\theta)^{2(n-1)} ,
$$
where the last equality is due to \cite[Lemma~2.1]{BCL:06a}.
\end{proof}

\begin{lemma}\label{le:condprobound}
For any $t>0$ we have
$$
   \Prob_{A\sim \tilde{\rho}} \Big\{ \|A^{-1}\| \ge t \Big\}
   \ \le\ \frac{e^2 (n+1)^2}{16\s^4}\, \frac1{t^4} .
$$
\end{lemma}

\begin{proof}
We use an idea in Sankar et al.~\cite[\S3]{sast:03}.
For any invertible $A\in\C^{n\times n}$ there exists
$u\in \bS^{n-1}$
such that $\|A^{-1}u\|=\|A^{-1}\|$.
For almost all~$A$, the vector $u$ is uniquely determined up
to a scaling factor $\theta$ of modulus~1. We shall denote
by $u_A$ a representative of such $u$.

The following is an easy consequence of the singular value
decomposition of~$\|A^{-1}\|$:
for any $v\in \bS^{n-1}$ we have
\begin{equation}\label{eq:Auv}
 \|A^{-1}v\| \ \ge\ \|A^{-1}\|\cdot |u_A\transp \, v| .
\end{equation}
We choose now a random pair $(A,v)$ with $A$ following
the law $\tilde{\rho}$ and,
independently, $v\in \bS^{n-1}$ from the uniform distribution.
Lemma~\ref{le:tbrho} implies that
$$
   \Prob_{A,v} \bigg\{ \|A^{-1}v\| \ge
   t \sqrt{\frac{2}{n+1}} \bigg\}
   \ \le\ \frac{(n+1)^2}{16\s^4 t^4}.
$$
On the other hand, we have by \eqref{eq:Auv}
\begin{eqnarray*}
 \lefteqn{\Prob_{A,v} \Big\{ \|A^{-1}v\|
  \ge t \sqrt{2/(n+1)} \Big\}}\\
  &\ge & \Prob_{A,v} \Big\{ \|A^{-1}\| \ge t \ \mbox{ and }\
        |u_A\transp \, v| \ge \sqrt{2/(n+1)} \Big\} \\
  &\ge & \Prob_{A} \Big\{ \|A^{-1}\| \ge t \Big\}\
      \Prob_{A,v} \Big\{ |u_A\transp \, v| \ge \sqrt{2/(n+1)}\
      \Big|\ \|A^{-1}\| \ge t \Big\}.
\end{eqnarray*}
Lemma~\ref{le:tvb} tells us that for any fixed $u\in \bS^{n-1}$
we have
$$
  \Prob_{v} \Big\{ |u\transp \, v| \ge \sqrt{2/(n+1)}\Big\}
 = (1 - 2/(n+1))^{n-1} \ge e^{-2},
$$
the last inequality as
$(\frac{n+1}{n-1})^{n-1}
= (1 + \frac{2}{n-1})^{n-1} \le e^2$.
We thus obtain
$$
 \Prob_{A} \Big\{ \|A^{-1}\| \ge t \Big\}
  \le e^2\, \Prob_{A,v} \bigg\{ \|A^{-1}v\| \ge
  t \sqrt{\frac{2}{n+1}} \bigg\}
  \le \frac{e^2(n+1)^2}{16\s^4 t^4},
$$
as claimed.
\end{proof}


\begin{proof}[Proof of Proposition~\ref{pro:fiberE}]
By Lemma~\ref{le:condprobound} we obtain, for any $T_0>0$,
\begin{eqnarray*}
  \lefteqn{\E\big(\|A^{-1}\|_\op^2\big)
    =\int_0^\infty \Prob\big\{\|A^{-1}\|_\op^2\geq T\big\}\,dT } \\
  &\ \leq\ T_0 + \int_{T_0}^\infty \Prob\big\{\|A^{-1}\|_\op^2\geq T\big\}\, dT
    \ \leq\   T_0 + \frac{e^2 (n+1)^2}{16\s^4}\, \frac1{T_0} ,
\end{eqnarray*}
using $\int_{T_0}^\infty T^{-2}\, dT   =  T_0^{-1}$.
Now choose $T_0 =\frac{e(n+1)}{4\s^2}$.
\end{proof}

\section{Smoothed Analysis of the Mean Square Condition Number}\label{se:red}

The goal of this section is to accomplish the proof of
Theorem~\ref{th:mu2-bound}.

\subsection{Orthogonal decompositions of $\Hd$}

For reasons to become clear soon
we have to distinguish points in $\proj^n$
from their representatives~$\z$ in the
sphere $\bS^n=\{ \z\in\C^{n+1} \mid \|\z\| = 1\}$.

For $\z\in\bS^n$ we consider the subspace $R_\z$ of~$\Hd$
consisting of all systems $h$ that vanish at $\z$ of higher order:
$$
  R_\z := \{h\in\Hd\mid h(\zeta)=0, Dh(\zeta)=0\} .
$$
We further decompose the orthogonal complement $R_\z^\perp$
of $R_\z$ in~$\Hd$
(defined with respect to the Bombieri-Weyl Hermitian inner product).
Let $L_\z$ denote the subspace of~$R_\z^\perp$ consisting of the systems
vanishing at~$\z$ and let $C_\z$ denote its orthogonal complement in $R_\z^\perp$.
Then we have an orthogonal decomposition
\begin{equation}\label{eq:odecomp}
 \Hd = C_\z \oplus L_\z \oplus R_\z
\end{equation}
parameterized by $\z\in\bS^n$.

\begin{lemma}\label{le:odecomp}
The space $C_\z$ consists of the systems
$(c_i\langle X,\zeta\rangle^{d_i})$ with $c_i\in\C$.
The space $L_\z$ consists of the systems
$$
 g = (\sqrt{d_i}\, \langle X,\zeta\rangle^{d_i-1}\ell_i)
$$
where $\ell_i$ is a linear form vanishing at~$\z$.
Moreover,
if $\ell_i = \sum_{j=0}^n m_{ij} X_j$
with $M=(m_{ij})$,
then $\|g\| = \|M\|_F$.
\end{lemma}

\begin{proof}
By unitary invariance it suffices to verify the assertions
in the case $\zeta=(1,0,\ldots,0)$.
In this case this follows easily from the definition of the
Bombieri-Weyl inner product.
\end{proof}

The Bombieri-Weyl inner product on~$\Hd$ and the standard metric on
the sphere $\bS^n$ define a Riemannian metric on $\Hd\times\bS^n$
on which the unitary group $\cU(n+1)$ operates isometrically.
The ``lifting''
$$
   \hV :=\{(q,\z)\in\Hd\times\bS^n \mid q(\z) = 0 \}
$$
of the solution variety~$\Vp$ is easily seen to be a
$\cU(n+1)$-invariant Riemannian submanifold of $\Hd\times\bS^n$.

The projection $\pi_2\colon\hV\to\bS^n, (q,\z)\mapsto\z$
defines a vector bundle with
fibers $V_\z:=\pi_2^{-1}(\z)$.
In fact, \eqref{eq:odecomp} can be interpreted as an orthogonal
decomposition of the trivial Hermitian vector bundle $\Hd\times\bS^n\to\bS^n$
into subbundles $C$, $L$, and $R$ over $\bS^n$.
Moreover, the vector bundle~$V$ is the orthogonal sum of~$L$ and $R$: we have
$V_\z = L_\z\oplus R_\z$ for all~$\z$.

In the special case where all the degrees~$d_i$ are one,
$\Hd$ can be identified with the space
$\cM:=\C^{n\times (n+1)}$ of matrices and
the solution manifold~$\hV$ specializes to the manifold
$$
    \hW:=\big\{\big(M,\zeta)\in \cM\times\bS^n\mid M\zeta=0\big\}.
$$
The map $\pi_2$ specializes to the vector bundle
$p_2\colon\hW\to\bS^n, (M,\z)\mapsto \z$
with the fibers
$$
 W_\z :=\{M\in\cM \mid M\z=0\} .
$$
Lemma~\ref{le:odecomp} tells us that for each $\z$
we have {\em isometrical} linear maps
\begin{equation}\label{eq:isomaps}
 \mbox{$W_\z \to L_\z,\ M\mapsto g_{M,\z} :=
    \big(\sqrt{d_i}\, \langle X,\zeta\rangle^{d_i-1}\sum_j m_{ij} X_j\big)$.}
\end{equation}
In other words, the Hermitian vector bundles $W$ and $L$ over $\bS^n$
are isometric.
The fact that the map~\eqref{eq:isomaps}
depends on the choice of the representative of~$\z$
forces us to work over $\bS^n$ instead over $\proj^n$.
(All other notions introduced so far only depend on the base point in $\proj^n$.)

We compose the orthogonal bundle projection
$V_\z=L_\z \oplus R_\z\to L_\z$
with the bundle isometry $L_\z\simeq W_\z$
obtaining the map of vector bundles
\begin{equation}\label{eq:Psi}
 \Psi\colon V \to W,\ (g_{M,\z} + h,\z) \mapsto (M,\z)
\end{equation}
with fibers
$\Psi^{-1}(M,\z)$ 
isometric to $R_\z$.

\begin{lemma}\label{le:psi}
We have $\Psi(q,\z) = (\Delta^{-1} Dq(\zeta),\z)$ where
$\Delta:= \diag(\sqrt{d_i})$.
\end{lemma}

\begin{proof}
Let $(q,\z)\in V$ and $(M,\z):=\Psi(q,\z)$.
Then we have the decomposition
$q = 0 + g_{M,\z} + h \in C_\z \oplus L_\z \oplus R_\z$.
It is easily checked that $Dg_{M,\z}(\z) = \Delta M$.
Since $Dq(\z) = Dg_{M,\z}(\z)$ we obtain
$M=\Delta^{-1}Dq(\z)$.
\end{proof}

The lemma shows that the condition number~$\mun(q,\z)$
(cf.~\S\ref{se:cond-num})
can be described in terms of $\Psi$ as follows:
\begin{equation}\label{eq:munrr}
 \frac{\mun(q,\z)}{\|q\|} \ =\ \|M^\dagger\| ,\quad
 \mbox{where $(M,\z)=\Psi(q,\z)$.}
\end{equation}

\subsection{Outline of proof of Theorem~\ref{th:mu2-bound}}\label{se:induced}

Let $\rho_{\Hd}$ denote the density of the Gaussian $N(\oq,\s^2\Id)$ on $\Hd$,
where $\oq\in\Hd$ and $\s>0$.
For fixed $\z\in\bS^n$ we decompose the mean~$\oq$ as
$$
\oq = \ok_\z + \og_\z + \oh_\z \in C_\z \oplus L_\z \oplus R_\z
$$
according to~\eqref{eq:odecomp}.
If we denote by
$\rho_{C_\z}$, $\rho_{L_\z}$, and $\rho_{R_\z}$ the densities of the
Gaussian distributions in the spaces $C_\z$, $L_\z$, and $R_\z$ with
covariance matrices $\s^2\Id$ and means $\ok_\z,\oM_\z$, and $\oh_\z$,
respectively, then the density $\rho_{\Hd}$ factors as
\begin{equation}\label{eq:dfactor}
\rho_{\Hd}(k + g + h) = \rho_{C_\z}(k)\cdot \rho_{L_\z}(g)\cdot \rho_{R_\z}(h).
\end{equation}
The Gaussian density~$\rho_{L_\z}$ on $L_\z$
induces a Gaussian density~$\rho_{W_\z}$ on the fiber~$W_\z$
with the covariance matrix $\s^2\Id$
via the isometrical linear map~\eqref{eq:isomaps}, so
$\rho_{W_\z}(M) = \rho_{L_\z}(g_{M,\z})$.

We derive now from the given Gaussian distribution $\rho_{\Hd}$ on $\Hd$
a probability distribution on~$V$ as follows
(naturally extending $\rhozero$ introduced in \S\ref{sec:R&C}).
Think of choosing $(q,\z)$ at random from~$\hV$
by first choosing $q\in\Hd$ from $N(\oq,\s^2\Id)$, then
choosing one of its~$\Dn$~zeros~$[\z]\in\proj^n$ at random from the uniform
distribution on $\{1,\ldots,\Dn\}$, and finally choosing a representative~$\z$
in the unit circle $[\z]\cap\bS^n$ uniformly at random.
(An explicit expression of the corresponding probability density~$\rho_\hV$ on $\hV$
is given in~\eqref{eq:dens_V}.)

The plan to show Theorem~\ref{th:mu2-bound} is as follows.
The forthcoming Lemma~\ref{le:rho1} tells us that
\begin{equation}\label{eq:HV}
  \E_{\Hd}\Big(\frac{\mu_2^2(q)}{\|q\|^2} \Big) \ =\
  \E_\hV \Big(\frac{\mun^2(q,\z)}{\|q\|^2} \Big)
\end{equation}
where $\E_\Hd$ and $\E_\hV$
refer to the expectations with respect to
the distribution $N(\oq,\s^2\Id)$ on $\Hd$ and
the probability density $\rho_\hV$  on $\hV$, respectively.
Moreover, by Equation~\eqref{eq:munrr},
$$
 \E_\hV \Big(\frac{\mun^2(q,\z)}{\|q\|^2} \Big) \ =\
 \E_{\cM}\big(\|M^\dagger\|^2\big) ,
$$
where $\E_\cM$ denotes the expectation with respect to
the pushforward density~$\rho_\cM$ of $\rho_V$ with respect
to the map~$p_1\circ\Psi\colon V\to\cM$
(for more on pushforwards see~\S\ref{se:coarea}).

Of course, we need to better understand the density $\rho_\cM$.
Let $M\in\cM$ be of rank~$n$ and $\z\in\bS^n$ with $M\z=0$.
The following formula
\begin{equation}\label{eq:heuri}
\rho_\cM (M) = \rho_{C_\z}(0) \cdot
 \frac1{2\pi}\,\int_{\lambda\in S^1} \rho_{W_{\lambda\z}}(M)\, dS^1 .
\end{equation}
can be heuristically explained as follows.
We decompose a random~$q\in\Hd$ according to the decomposition
$\Hd=C_\z\oplus L_\z\oplus R_\z$ as $q = k + g + h$.
Choose $\lambda\in\C$ with $|\lambda|=1$ uniformly at random in the unit circle.
Then we have $\Psi(q,\lambda\z)=(M,\lambda\z)$ iff $k=0$ and $g$ is mapped
to $M$ under the isometry in~\eqref{eq:isomaps}.
The probability density for the event $k=0$ equals $\rho_{C_\z}(0)$.
The second event, conditioned on~$\lambda$,
has the probability density
$\rho_{W_{\lambda\z}}(M)$.

By general principles (cf.~\S\ref{se:coarea}) we have
\begin{equation}\label{eq:E-E}
\E_\cM\big(\|M^\dagger\|^2_\op\big) \ =\
\E_{\z\sim\rho_{\bS^n}}\Big( \E_{M\sim \tilde{\rho}_{W_\z}}
  \big( \|M^\dagger\|^2_\op\big)  \Big) ,
\end{equation}
where $\rho_{\bS^n}$ is the pushforward density of~$\rho_V$
with respect to $p_2\circ\Psi\colon V\to\bS^n$
and $\tilde{\rho}_{W_\z}$ denotes a
``conditional density'' on the fiber $W_\z$.
This conditional density turns out to be of the form
\begin{equation}\label{eq:form}
\tilde{\rho}_{W_\z}(M) = c_\z^{-1}\cdot \det(M M^*)\, \rho_{W_\z}(M),
\end{equation}
($c_\z$ denoting a normalization factor).
In the case $\z=(1,0,\ldots,0)$ we can identify $W_\z$ with
$\C^{n\times n}$ and $\tilde{\rho}_{W_\z}$ takes the form~\eqref{eq:rhotilde}
studied in Section~\ref{se:smMC}.
Proposition~\ref{pro:fiberE} and unitary invariance imply that
for all~$\z\in\bS^n$
\begin{equation}\label{eq:enaf}
\E_{M\sim \tilde{\rho}_{W_\z}}
  \big( \|M^\dagger\|^2_\op\big) \ \le\ \frac{e(n+1)}{2\s^2} .
\end{equation}
This implies by~\eqref{eq:E-E} that
$$
\E_\cM\big(\|M^\dagger\|^2_\op\big) \ \le\ \frac{e(n+1)}{2\s^2}
$$
and completes the outline of the proof of Theorem~\ref{th:mu2-bound}.

The formal proof of the stated facts \eqref{eq:heuri}--\eqref{eq:form}
is quite involved and
will be given in the remainder of this section.

\subsection{Coarea formula}\label{se:coarea}

We begin by recalling the coarea formula that tells us how
probability distributions on Riemannian manifolds transform.

Suppose that $X,Y$ are Riemannian manifolds of dimensions~$m$, $n$,
respectively such that $m\ge n$. Let $\varphi\colon X\to Y$ be
differentiable. By definition, the derivative $D\varphi(x)\colon
T_xX\to T_{\varphi(x)} Y$ at a regular point $x\in X$ is surjective.
Hence the restriction of $D\varphi(x)$ to the orthogonal complement of
its kernel yields a linear isomorphism. The absolute value of its
determinant is called the {\em normal Jacobian} of $\varphi$ at $x$
and denoted $\NJ\varphi(x)$. We set $\NJ\varphi(x):=0$ if $x$~is not a
regular point.  We note that the fiber $F_y:=\varphi^{-1}(y)$ is a
Riemannian submanifold of~$X$ of dimension $m-n$ if $y$ is a regular
value of~$\varphi$. Sard's lemma states that almost all $y\in Y$ are
regular values.

The following result is the coarea formula,
sometimes also called Fubini's Theorem for Riemannian manifolds.
A proof can be found e.g., in~\cite[Appendix]{howa:93}.

\begin{proposition}\label{pro:coarea}
Suppose that $X,Y$ are Riemannian manifolds of dimensions~$m$, $n$, respectively,
and let $\varphi\colon X\to Y$ be a surjective differentiable map.
Put $F_y=\varphi^{-1}(y)$.
Then we have for any function $\chi\colon X\to\R$ that is integrable
with respect to the volume measure of $X$ that
\begin{center}
$\displaystyle
  \int_X \chi \, dX= \int_{y\in Y} \left( \int_{F_y}
  \frac{\chi}{\NJ\varphi}\, dF_y\right) dY .
$
\end{center}
\vspace*{-0.8cm}{{\mbox{}\hfill\qed}}\vspace{0.6cm}
\end{proposition}

Now suppose that we are in the situation described in the statement
of Proposition~\ref{pro:coarea} and we
have a probability measure on $X$ with density $\rho_X$.
For a regular value $y\in Y$ we set
\begin{equation}\label{eq:pushf}
\rho_Y(y) = \int_{F_y} \frac{\rho_X}{\NJ\varphi}\, dF_y .
\end{equation}
The coarea formula implies that for all measurable sets $B\subseteq Y$ we have
$$
\int_{\varphi^{-1}(B)} \rho_X\, dX = \int_B \rho_Y\, dY.
$$
Hence $\rho_Y$ is a probability density on $Y$.
We call it the {\em pushforward} of $\rho_X$ with respect to $\varphi$.

For a regular value $y\in Y$ and $x\in F_y$ we define
\begin{equation}\label{eq:condd}
\rho_{F_y}(x) = \frac{\rho_X(x)}{\rho_Y(y)\NJ\varphi(x)} .
\end{equation}
Clearly, this defines a probability density on $F_y$.
The coarea formula implies that for all measurable
functions $\chi\colon X\to \R$
$$
   \int_X \chi\, \rho_X \, dX= \int_{y\in Y} \left(\int_{F_y} \chi\,
         \rho_{F_y}\, dF_y \right)\rho_Y(y)\, dY ,
$$
provided the left-hand integral exists.
Therefore, we can interpret $\rho_{F_y}$ as the
{\em density of the conditional distribution} of $x$ on the fiber $F_y$ and
briefly express the formula above in probabilistic terms as
\begin{equation}\label{eq:Eit}
   \E_{x\sim \rho_X}(\chi(x)) = \E_{y\sim \rho_Y}
   \big( \E_{x\sim \rho_{F_y}}(\chi(x)) \big) .
\end{equation}

To put these formulas at use in our context,
we must compute the normal Jacobians of some maps.

\subsection{Normal Jacobians}\label{se:NJ}

We start with a general comment.
Note that the $\R$-linear map $\C\to\C,z\mapsto \lambda z$ with
$\lambda\in\C$ has determinant~$|\lambda|^2$.
More generally, let $\varphi$ be an
endomorphism of a finite dimensional complex vector space.
Then $|\det\varphi|^2$ equals the determinant of $\varphi$,
seen as a $\R$-linear map.

We describe now the normal Jacobian of the
projection $p_1\colon W\to\cM$ following~\cite{Bez2}.

\begin{lemma}\label{le:NJp1}
We have
$\NJ p_1 (M,\z) =\prod_{i=1}^n (1 + \s_i^{-2})^{-1}$
where $\s_1,\ldots,\s_n$ are the singular values of~$M$.
\end{lemma}

\begin{proof}
First note that
$T_\z\bS^n =\{ \dot{\z}\in\C^{n+1} \mid \mathrm{Re}\langle \z,\dot{\z}\rangle =0\}$.
The tangent space $T_{(M,\z)}W$ consists of the
$(\dot{M},\dot{\z})\in\cM\times T_\z\bS^n$
such that $\dot{M}\z + M\dot{\z} = 0$.

By unitary invariance we may assume that $\z=(1,0,\ldots,0)$.
Then the first column of $M$ vanishes and we denote by
$A=[m_{ij}]\in\C^{n\times n}$ the remaining part of $M$.
W.l.o.g. we may assume that $A$ is invertible.
Further, let
$\dot{u}\in\C^n$ denote the first column of $\dot{M}$ and
$\dot{A}\in\C^{n\times n}$ its remaining part.
We may thus identify $T_{(M,\z)}W$ with the product $E \times \C^{n\times n}$
via $(\dot{M},\dot{\z}) \mapsto ((\dot{u},\dot{\z}),\dot{A})$,
where $E$ denotes the subspace
$$
 E:= \bigg\{ (\dot{u},\dot{\z}) \in\C^n\times \C^{n+1} \mid
  \dot{u}_i + \sum_{j=1}^n m_{ij} \dot{\z}_j = 0, 1\le i\le n,\,  
 \,  \dot{\z}_0 \in i \R \bigg\} .
$$
We also note that $E\simeq \graph(-A)\times i\R$.
The derivative of $p_1$ is described by the following commutative
diagram
$$
\begin{array}{ccc}
 T_{(M,\z)}W & \stackrel{\simeq}{\longrightarrow} &
   (\graph(-A)\times i\R)\times \C^{n\times n} \\
  \left.{\scriptstyle Dp_1(M,\z)}\rule{0mm}{5mm}
  \right\downarrow &
  & \left\downarrow {\scriptstyle \pr\times \id}
      \rule{0mm}{5mm}\right.\\
  \cM & \stackrel{\simeq}{\longrightarrow} &
  \C^n\times \C^{n\times n} ,
\end{array}
$$
where $\pr(\dot{u},\dot{\z}) = \dot{u}$.
Using the singular value decomposition we may assume that
$A=\diag(\s_1,\ldots,\s_n)$.
Then the pseudoinverse of the projection $\pr$ is given by
the $\R$-linear map
$$
 \varphi\colon\C^n\to \graph(-A),\,
 \dot{u} \mapsto (\dot{u}, -\s_1^{-1}\dot{u}_{1},\ldots,-\s_n^{-1}\dot{u}_{n}) .
$$
It is easy to see that
$\det\varphi =\prod_{i=1}^n (1 + \s_i^{-2})$.
To complete the proof we note that
$1/\NJ p_1 (M,\z)=\det\varphi$.
\end{proof}

We have already seen that the condition number~$\mun(q,\z)$
can be described in terms of the map~$\Psi$ introduced in~\eqref{eq:Psi}.
As a stepping stone towards the analysis of the normal Jacobian of~$\Psi$
we introduce now the related bundle map
$$
\Phi\colon \hV\to \hW,\ (q,\z)\mapsto (Dq(\z),\z),
$$
whose normal Jacobian turns out to be constant.
(This crucial observation is due to
Beltr\'an and Pardo in~\cite{BePa08b}.)

\begin{proposition}\label{pro:NJs}
We have $\NJ\Phi(q,\z) = \Dn^n$ for all $(q,\z)\in \hV$.
\end{proposition}

\begin{proof}

By unitary invariance we may assume without loss of generality that $\z=(1,0,\ldots,0)$.
If we write $N= (n_{ij}) = Dq(\z)\in\cM$ we must have $n_{i0}=0$ since $N\z=0$.
Moreover, according to the orthogonal decomposition~\eqref{eq:odecomp}
and Lemma~\ref{le:odecomp}, we have for $1\le i\le n$
$$
 q_i = X_0^{d_i-1}\sum_{j=1}^n n_{ij}X_j + h_i
$$
for some $h=(h_1,\ldots,h_n)\in R_\z$.
We further express $\dot{q}_i\in T_q\Hd=\Hd$ as
$$
 \dot{q}_i = \dot{u}_i X_0^{d_i}
  + \sqrt{d_i} X_0^{d_i-1}\sum_{j=1}^n\dot{a}_{ij}X_j + \dot{h}_i
$$
in terms of the coordinates $\dot{u} =(\dot{u}_i)\in\C^n$,
$\dot{A}=(\dot{a}_{ij})\in\C^{n\times n}$, and $\dot{h}=(\dot{h}_i)\in R_\z$.
The reason to put the factor $\sqrt{d_i}$ here is that
\begin{equation}\label{eq:BW}
\|\dot{q}\|^2= \sum_{i}|\dot{u}_{i}|^2 + \sum_{ij}|\dot{a}_{ij}|^2
  + \sum_i \|\dot{h}_i\|^2
\end{equation}
by the definition of the Bombieri-Weyl inner product.

The tangent space $T_{(q,\z)}V$ consists of the
$(\dot{q},\dot{\z})\in\Hd\times T_\z\bS^n$
such that $\dot{q}(\z) + N\dot{\z} = 0$,
see~\cite[\S10.3, Prop.~1]{bcss:95}.
This condition can be expressed in coordinates as
\begin{equation}\label{eq:tanc}
  \dot{u}_i + \sum_{j=1}^n n_{ij} \dot{\z}_j = 0,
  \quad \mbox{$i=1,\ldots,n$.}
\end{equation}
By~\eqref{eq:BW} the inner product on $T_{(q,\z)}V$ is given
by the standard inner product
in the chosen coordinates $\dot{u}_i,\dot{a}_{ij},\dot{\z}_j$
if $\dot{h}_i=0$.
Thinking of the description of $T_{(N,\z)}W$ given in the proof
of Lemma~\ref{le:NJp1},
we may therefore isometrically identify $T_{(q,\z)}V$ with the product
$T_{(N,\z)}W \times R_\z$ via
$(\dot{q},\dot{\z})\mapsto ((\dot{u},\dot{A},\dot{\z}),\dot{h})$.
The derivative of $\pi_1$ is then described by the commutative diagram
\begin{equation}\label{eq:idp}
\begin{array}{ccc}
 T_{(q,\z)}V & \stackrel{\simeq}{\longrightarrow}
   & T_{(N,\z)} W \times R_\z \\
  \left.{\scriptstyle D\pi_1(q,\z)}
  \rule{0mm}{5mm}\right\downarrow &
  & \left\downarrow { \scriptstyle Dp_1(N,\z) \times \id}
   \rule{0mm}{5mm}\right.\\
  \Hd & \stackrel{\simeq}{\longrightarrow} & \cM \times R_\z .
\end{array}
\end{equation}

We shall next calculate the derivative of $\Phi$.
For this, we will use the shorthand $\partial_kq$ for
the partial derivative $\partial_{X_k}q$, etc.
A short calculation yields, for $j>0$,
\begin{equation}\label{eq:derform}
 \partial_0\dot{q}_i(\z) = d_i \dot{u}_i,\quad
 \partial_j\dot{q}_i(\z) = \sqrt{d_i}\,\dot{a}_{ij},\quad
 \partial_{0j}^2q_i(\z) = (d_i-1)\, n_{ij} .
\end{equation}
Similarly, we obtain
$\partial_0 q_i(\z) = 0$ and $\partial_j q_i(\z) = n_{ij}$
for $j>0$.

The derivative of $D\Phi(q,\z)\colon T_{(q,\z)}V \to T_{(N,\z)}W$
is determined by
$$
D\Phi(q,\z)(\dot{q},\dot{\z}) = (\dot{N},\dot{\z}),
\mbox{ where }
 \dot{N} = D\dot{q}(\z) + D^2q(\z)(\dot{\z},\cdot) .
$$
Introducing the coordinates $\dot{N}=(\dot{n}_{ij})$ this can be written as
\begin{equation}\label{eq:starstar}
 \dot{n}_{ij} = \partial_j \dot{q}_i(\z) + \sum_{k=1}^n\partial^2_{jk}\dot{q}_i(\z)\,\dot{\z}_k .
\end{equation}
For $j>0$ this gives, using~\eqref{eq:derform},
\begin{equation}\label{eq:Der2}
  \dot{n}_{ij} = \sqrt{d_i}\,\dot{a}_{ij} + \sum_{k=1}^n\partial^2_{jk}\dot{q}_i(\z)\,\dot{\z}_k .
\end{equation}
For $j=0$ we obtain from~\eqref{eq:starstar}, using~\eqref{eq:derform}
and~\eqref{eq:tanc},
\begin{equation}\label{eq:Der1}
 \dot{n}_{i0} = \partial_0\dot{q}_i(\z) + \sum_{k=1}^n\partial^2_{0k}\dot{q}_i(\z)\,\dot{\z}_k
  = d_i\dot{u}_i + (d_i-1)\sum_{k=1}^n n_{ik}\,\dot{\z}_k
= \dot{u}_i .
\end{equation}
Note the crucial cancellation taking place here!

From~\eqref{eq:Der2} and \eqref{eq:Der1}
we see that the kernel~$K$ of $D\Phi(q,\z)$ is determined
by the conditions $\dot{\z}=0,\dot{u}=0,\dot{A}=0$.
Hence, recalling $T_{(q,\z)}V \simeq T_{(N,\z)}W \times R_\z$,
we have $K\simeq 0 \times R_\z$ and
$K^\perp\simeq T_{(N,\z)} W \times 0$.
Moreover,
as in the proof of Lemma~\ref{le:NJp1} (but replacing $M$ by $N$)
we write
$$
 E:= \bigg\{ (\dot{u},\dot{\z}) \in\C^n\times \C^{n+1} \mid
 \dot{u}_i + \sum_{j=1}^n n_{ij} \dot{\z}_j = 0, 1\le i\le n,\,
     \dot{\z}_0 \in i \R \bigg\}
$$
and identify $T_{(N,\z)} W$ with $E\times \C^{n\times n}$.
Using this identification of spaces,
\eqref{eq:Der2} and \eqref{eq:Der1} imply that
$D\Phi(q,\z)_{K^\perp}$ has the following structure:
\begin{eqnarray*}
 D\Phi(q,\z)_{K^\perp}\colon E\times \C^{n\times n}
&\to& E\times \C^{n\times n},\\
   ((\dot{u},\dot{\z}),\dot{A}) &\mapsto&
   ((\dot{u},\dot{\z}), \lambda(\dot{A}) + \rho(\dot{\z})) ,
\end{eqnarray*}
where the linear map
$\lambda\colon\C^{n\times n}\to \C^{n\times n},
\dot{A} \mapsto (\sqrt{d_i}\,\dot{a}_{ij})$,
multiplies the $i$th row of $\dot{A}$ with $\sqrt{d_i}$ and
$\rho\colon\C^{n+1}\to \C^{n\times n}$ is given by
$\rho(\dot{\z})_{ij} = \sum_{k=1}^n\partial^2_{jk}\dot{q}_i(\z)
\,\dot{\z}_k$.

By definition we have
$\NJ\Phi(q,\z) = | \det D\Phi(q,\z)_{|K^\perp}|$.
The triangular form of $D\Phi(q,\z)_{K^\perp}$ shown above implies that
$| \det D\Phi(q,\z)_{|K^\perp}| = \det\lambda$.
Finally, using the diagonal form of $\lambda$,
we obtain $\det\lambda =\prod_{i=1}^n \sqrt{d_i}^2 = \Dn^n$,
which completes the proof.
\end{proof}

\begin{remark}
An inspection of the proof of Proposition~\ref{pro:NJs} reveals that
the second order derivatives occuring
in $D\Phi$ do not have any impact on the normal Jacobian $\NJ\Phi$.
Its value~$\Dn^n$ occurs as a result of the chosen Bombieri-Weyl inner product on $\Hd$.
With respect to the naive inner product on $\Hd$
(where the monomials form an orthonormal basis),
the normal Jacobian of $\Phi$ at $(q,\z)$ would be equal to one
at $\z=(1,0,\ldots,0)$.
However unitary invariance would not hold and the
normal Jacobian would take different values elsewhere.
\end{remark}

Before proceding we note the following consequence
of Equation~\eqref{eq:idp}:
\begin{equation}\label{eq:Npip}
\NJ \pi_1(q,\z) = \NJ p_1 (N,\z) \quad\mbox{where $N=Dq(\z)$}.
\end{equation}

The normal Jacobian of the map $\Psi\colon V\to W$ is not constant
and takes a more complicated form in terms of the normal
Jacobians of the projection $p_1\colon W\to\cM$.
For obtaining an expression for $\NJ\Psi$
we need the following lemma.

\begin{lemma}\label{le:NJsc}
The scaling map
$\gamma\colon W\to W,(N,\z)\mapsto (M,\z)$ with $M=\Delta^{-1}N$ of rank~$n$
satisfies
$$
\det D\gamma(N,\z) = \frac1{\Dn^{n+1}}\cdot\frac{\NJ \ps_1 (N,\z)}{\NJ \ps_1 (M,\z) }.
$$
\end{lemma}

\begin{proof}
If $W_\proj$ denotes the solution variety in
$\cM\times\proj^n$ analogous to~$W$
we have
$T_{(M,\z)}W = T_{(M,\z)}W_\proj \oplus \R i\z$.
Let $p_1'\colon W_\proj\to\cM$ denote the projection.
The derivative $D\gamma_\proj(N,\z)$
of the corresponding scaling map
$\gamma_\proj\colon W_\proj\to W_\proj$ is determined by
the commutative diagram
$$
\begin{array}{rcl}
 T_{(N,\z)}W_\proj & \stackrel{D\gamma_\proj(N,\z)}
   {\longrightarrow} & T_{(M,\z)}W_\proj \\
  \left.{\scriptstyle Dp_1'(N,\z)}
  \rule{0mm}{5mm}\right\downarrow &
  & \left\downarrow {\scriptstyle Dp_1'(M,\z)}
  \rule{0mm}{5mm}\right.\\
  \cM & \stackrel{\mathrm{sc}}{\longrightarrow} & \cM
\end{array}
$$
where the vertical arrows are linear isomorphisms.
The assertion follows by observing that
$\NJ p_1(N,\z) = \det Dp_1'(N,\z)$,
$\NJ \gamma (N,\z) = \det D \gamma_\proj (N,\z)$,
and using that the $\R$-linear map
$
  \mathrm{sc}\colon\cM\to\cM, N\mapsto M=\Delta^{-1}N
$
has the determinant $1/\Dn^{n+1}$.
\end{proof}

Proposition~\ref{pro:NJs} combined with Lemma~\ref{le:NJsc}
immediately gives
\begin{equation}\label{eq:NJPsi}
 \NJ\Psi(q,\z) = \frac1{\Dn}\cdot \frac{\NJ \ps_1 (N,\z)}{\NJ \ps_1 (M,\z)}
\end{equation}
for $N=Dq(\z)$, $M=\Delta^{-1}N$.

\subsection{Induced probability distributions}\label{se:ind-dist}

By B\'ezout's theorem, the fiber $\hV(q)$ of the projection
$\pi_1\colon\hV\to \Hd$ at $q\in\Hd$ is a disjoint union of
$\Dn=d_1\cdots d_n$ unit circles and therefore has the volume $2\pi\Dn$,
provided $q$ does not lie in the discriminant variety.

Recall that $\rho_{\Hd}$ denotes the density of
the Gaussian distribution $N(\oq,\s^2\Id)$ for
fixed $\oq\in\Hd$ and $\s>0$ and
$\E_\Hd$ stands for expectation taken with respect to that density.
We associate with $\rho_{\Hd}$ the function $\rho_\hV:\hV\to\R$ defined by
\begin{equation}\label{eq:dens_V}
  \rho_\hV(q,\zeta):=\frac{1}{2\pi\Dn}\,\rho_{\Hd}(q)\,
  \NJ\pi_1(q,\zeta) .
\end{equation}
The next result shows that $\rho_V$ is the probability density
function of the distribution on $\hV$ we described in
\S\ref{se:induced}.

\begin{lemma}\label{le:rho1}
\begin{enumerate}
\item The function $\rho_\hV$ is a probability density on $\hV$.

\item The expectation of a function
$\varphi\colon\hV\to \R$ with respect to~$\rho_\hV$ can be expressed as
$\E_\hV(\varphi) = \E_{\Hd}(\varphi_{av})$, where
$$
 \varphi_{\av}(q) :=  \frac1{2\pi\Dn}\int_{\hV(q)} \varphi\, d\hV(q) .
$$

\item The pushforward of $\rho_\hV$ with respect to $\pi_1\colon V\to\Hd$
equals $\rho_\Hd$.

\item For $q\not\in\Sigma$, the conditional density on the fiber $\hV(q)$  is the
density of the uniform distribution on $\hV(q)$.

\item The probability density $\rhozero$ on $\Vp$ introduced in \S\ref{sec:R&C}
is obtained from the density $\rho_\hV$ in the case $\oq=0$,
$\sigma=1$ as the pushforward under the canonical map $\hV\to \Vp,
(f,\z)\mapsto (f,[\z])$. Explicitly, we have
$$
\rhozero(q,[\z]) = \frac1{\Dn} \frac1{(2\pi)^{N}}\, e^{-\frac12
\|q\|^2} \NJ\pi_1(q,\z) .
$$
\end{enumerate}
\end{lemma}

\begin{proof}
The coarea formula (Proposition~\ref{pro:coarea}) applied to
$\pi_1\colon\hV\to\Hd$ implies
\begin{eqnarray*}
  \int_\hV\varphi\, \rho_\hV d\hV
  &=& \int_{q\in\Hd}\Big(\int_{\z\in \hV(q)}
  \varphi(q,\z) \frac{\rho_\hV(q,\z)}{\NJ\pi_1(q,\zeta)}\,d\hV(q) \Big) \,d\Hd \\
&=&\int_{q\in\Hd} \varphi_{\av}(q)\, \rho_{\Hd}(q)\, d\Hd.
\end{eqnarray*}
Taking $\varphi=1$ reveals that $\rho_\hV$ is a density, proving the
first assertion. The above formula also shows the second assertion.

By Equation~\eqref{eq:pushf} the pushforward density $\rho$ of $\rho_V$
with respect to $\pi_1$ satisfies
$$
   \rho(q) = \int_{\z\in\hV(q)} \frac{\rho_\hV(q,\z)}{\NJ\pi_1(q,\z)}\,
  d\hV(q) =\rho_\Hd(q),
$$
as $\int dV(q) = 2\pi\Dn$.
This shows the third assertion. By \eqref{eq:condd} the conditional
density satisfies
$$
    \rho_{\hV(q)}(q) =
  \frac{\rho_\hV(q,\z)}{\rho_\Hd(q)\,\NJ\pi_1(q,\z)}= \frac1{2\pi\Dn},
$$
which shows the fourth assertion. The fifth assertion is trivial.
\end{proof}

We can now determine the various probability distributions induced by~$\rho_V$.

\begin{proposition}\label{pro:rhoW-1}
We have
$$
 \frac{\rho_\hV}{\NJ\Psi}(g_{M,\z} + h,\z)
 = \rho_\hW(M,\z)\cdot \rho_{R_\z}(h),
$$
where the pushforward density $\rho_W$ of $\rho_\hV$ with respect to~$\Psi\colon V\to W$ satisfies
$$
 \rho_\hW(M,\z) = \frac1{2\pi}\, \rho_{C_\z}(0)\cdot \rho_{W_\z}(M)\cdot \NJ p_1(M,\z).
$$
\end{proposition}

\begin{proof}
Using the factorization of Gaussians~\eqref{eq:dfactor} and Equation~\eqref{eq:Npip},
the density~$\rho_\hV$ can be written as
\begin{equation*}\label{eq:frhoV}
 \rho_\hV(g_{M,\z} + h,\z) = \frac1{2\pi\Dn}\,
  \rho_{C_\z}(0)\, \rho_{W_\z}(M)\, \rho_{R_\z}(h)\,
  \NJ p_1(N,\zeta),
\end{equation*}
where $N=\Delta M$.
It follows from~\eqref{eq:NJPsi} that
\begin{equation}\label{eq:istep}
 \frac{\rho_\hV}{\NJ\Psi}(g_{M,\z} + h,\z)
  = \frac1{2\pi}\, \rho_{C_\z}(0)\, \rho_{W_\z}(M)\, \rho_{R_\z}(h)\,
  \NJ p_1(M,\zeta).
\end{equation}
This implies, using~\eqref{eq:pushf} for
$\Psi:V\to W$ and the isometry $\Psi^{-1}(M,\z)\simeq R_\z$
for the fiber at $\z$, that
\begin{eqnarray*}
 \rho_W(M,\z) &=& \int_{h\in R_\z}
 \frac{\rho_\hV}{\NJ\Psi}(g_{M,\z} + h,\z)\, dR_\z\\
 &=& \frac1{2\pi}\, \rho_{C_\z}(0)\cdot \rho_{W_\z}(M)
    \cdot \NJ p_1(M,\z) \int_{h\in R_\z}\rho_{R_\z}(h)dR_\z \\
 &=& \frac1{2\pi}\, \rho_{C_\z}(0)\cdot \rho_{W_\z}(M)
    \cdot \NJ p_1(M,\z)
\end{eqnarray*}
as claimed. Replacing in~\eqref{eq:istep} we therefore obtain
\begin{equation*}
 \frac{\rho_\hV}{\NJ\Psi}(g_{M,\z} + h,\z)
 = \rho_\hW(M,\z)\ \rho_{R_\z}(h). \qedhere
\end{equation*}
\end{proof}

The claimed formula~\eqref{eq:heuri} for the pushforward
density~$\rho_\cM$
of $\rho_W$ with respect to $p_1\colon W\to\cM$ immediately follows
from Proposition~\ref{pro:rhoW-1} by integrating
$\frac{\rho_W}{\NJ p_1}$ over the fibers of $p_1$.

\begin{lemma}\label{pro:rhoW-2}
Let $c_\z$ denote the expectation of $\det (MM^*)$ with respect to~$\rho_{W_\z}$.
We have
$$
 \frac{\rho_W}{\NJ p_2}(M,\z) = \rho_{\bS^n}(\z)\cdot \tilde{\rho}_{W_\z}(M),
$$
where $\rho_{\bS^n}(\z) = \frac{c_\z}{2\pi} \rho_{C_\z}(0)$ is
the pushforward density of $\rho_W$ with respect to
$p_2\colon W\to\bS^n$,
and where the conditional density $\tilde{\rho}_{W_\z}$ on the fiber $W_\z$ of~$\ps_2$
is given by
\begin{equation*}
   \tilde{\rho}_{W_\z}(M) = c_\z^{-1}\cdot \det (MM^*) \rho_{W_\z}(M) .
\end{equation*}
\end{lemma}

\begin{proof}
In~\cite{Bez2} (see also \cite[Section~13.2, Lemmas~2-3]{bcss:95})
it is shown that
\begin{equation}\label{eq:NJq}
 \frac{\NJ p_1}{\NJ p_2}(M,\z) = \det (M M^*).
\end{equation}
Combining this with Proposition~\ref{pro:rhoW-1} we get
$$
 \frac{\rho_W}{\NJ p_2}(M,\z)
  = \frac1{2\pi}\, \rho_{C_\z}(0)\cdot \rho_{W_\z}(M)\cdot \det(MM^*).
$$
Integrating over $W_\z$ we get
$\rho_{\bS^n}(\z) = \frac1{2\pi}\, \rho_{C_\z}(0)\cdot c_\z$,
and finally (cf.~\eqref{eq:condd})
$$
 \tilde{\rho}_{W_\z}(M) = \frac{\rho_W(M,\z)}{\rho_{\bS^n}(\z)\,\NJ p_2(M,\z)}
  = c_\z^{-1}\cdot \rho_{W_\z}(M)\cdot \det(MM^*)
$$
as claimed.
\end{proof}

This lemma shows that the conditional density $\tilde{\rho}_{W_\z}$
has the form stated in~\eqref{eq:form} and therefore completes the proof of
Theorem~\ref{th:mu2-bound}.

\subsection{Expected number of real zeros}

As a further illustration of the interplay of Gaussians with
the coarea formula in our setting, we give
a simplified proof of one of the main results of~\cite{Bez2}.
This subsection is not needed for understanding the remainder of the paper.

Our developments so far took place over the complex numbers~$\C$,
but much of what has been said carries over the situation over~$\R$.
However, we note that algorithm \ALH\ would not work over~$\R$
since the lifting of the segment $E_{f,g}$ will likely
contain a multiple zero
(over $\C$ this happens with probability zero since the
real codimension of the discriminant variety equals two).

Let $\HdR$ denote the space of real polynomial systems in $\Hd$
endowed with the Bombieri-Weyl inner product.
The standard Gaussian distribution on $\HdR$ is well-defined
and we denote its density with $\rho_{\HdR}$.

\begin{corollary}\label{cor:average-real}
The average number of zeros of a standard Gaussian random $f\in\HdR$
in the real projective space~$\proj^n(\R)$
equals $\sqrt{\Dn}$.
\end{corollary}

\begin{proof}
Let $\chi(q)$ denote the number of real zeros in
$\proj^n(\R)$ of $q\in\HdR$.
Thus the number of real zeros
in the sphere~$S^n=S(\R^{n+1})$ equals $2\chi(q)$.
The real solution variety
$V_\R\subseteq\HdR\times S^n$
is defined in the obvious way and so is
$W_\R\subseteq \cM_\R\times S^n$,
where $\cM_\R=\R^{n\times (n+1)}$.

The same proof as for Proposition~\ref{pro:NJs}
shows that the normal Jacobian of the map
$\Phi_\R\colon V_\R\to W_\R, (q,\z)\mapsto (Dq(\z),\z)$
has the constant value $\Dn^{n/2}$
(the~$2$ in the exponent due to the considerations opening \S\ref{se:NJ}).

Applying the coarea formula to the
projection~$\pi_1\colon V_\R\to\HdR$ yields
\begin{eqnarray*}
\int_{\HdR} \chi \,\rho_{\HdR}\, d\HdR
   &=& \int_{q\in\HdR} \rho_{\HdR}(q)\, \frac12 \int_{\pi_1^{-1}(q)} d\pi_1^{-1}(q)\, d\HdR \\
   &=&\int_{V_\R} \frac12 \rho_{\HdR} \,\NJ\pi_1\ dV_\R .
\end{eqnarray*}

We can factor the standard Gaussian~$\rho_{\Hd}$
into standard Gaussian densities~$\rho_{C_\z}$ and $\rho_{L_\z}$ on $C_\z$ and $L_\z$, respectively,
as it was done in \S\ref{eq:dfactor} over $\C$
(denoting them by the same symbol will not cause any confusion).
We also have an isometry $W_\z\to L_\z$  as in~\eqref{eq:isomaps}
and $\rho_{L_\z}$ induces the standard Gaussian density $\rho_{W_\z}$ on $W_\z$.
The fiber of $\Phi_\R\colon V_\R\to W_\R,(q,\z)\mapsto (N,\z)$ over $(N,\z)$ has the form
$\Phi_\R^{-1}(N,\z)=\{(g_{M,\z}+h,\z) \mid h\in R_\z \}$ where $M=\Delta^{-1} N$,
cf.~Lemma~\ref{le:psi}.
We therefore have
$\rho_{\HdR}(g_{M,\z} + h) = \rho_{C_\z}(0)\, \rho_{W_\z}(M)\, \rho_{R_\z}(h)$.

The coarea formula applied to~$\Phi_\R\colon V_\R\to W_\R$, using Equation~\eqref{eq:Npip}, yields
\begin{eqnarray*}
\lefteqn{\int_{V_\R} \frac12\, \rho_{\HdR}\, \NJ\pi_1\ dV_\R } \\
 &=&\frac1{2\,\NJ\Phi_\R} \int_{(N,\z)\in W_\R} \rho_{C_\z}(0)\, \rho_{W_\z}(M)\, \NJ p_1(N,\z)
  \int_{h\in R_\z}\rho_{R_\z}(h) \, dR_\z \, dW_\R \\
  &=& \frac1{2\,\NJ\Phi_\R} \int_{(N,\z)\in W_\R} \rho_{C_\z}(0)\, \rho_{W_\z}(M)\, \NJ p_1(N,\z)\,dW_\R .
\end{eqnarray*}
Applying the coarea formula to the projection $p_1\colon W_\R\to \cM_\R$,
we can simplify the above to
\begin{eqnarray*} \notag
 & & \frac1{\NJ\Phi_\R} \int_{N\in \cM_\R} \rho_{C_\z}(0)\, \rho_{W_\z}(M)\,
  \frac1{2} \int_{\z\in p_1^{-1}(N)} dp_1^{-1}(N) \, d\cM_\R \\ \notag
 &=& \frac1{\NJ\Phi_\R} \int_{N\in \cM_\R} \rho_{C_\z}(0)\, \rho_{W_\z}(M)\, d\cM_\R \\ \label{eq:last}
 &=& \frac{\Dn^{\frac{n+1}{2}}}{\NJ\Phi_\R} \int_{M\in \cM_\R} \rho_{C_\z}(0)\, \rho_{W_\z}(M)\, d\cM_\R ,
\end{eqnarray*}
where the last equality is due to
the change of variables $\cM_\R\to\cM_\R, N\mapsto M$
that has the Jacobian determinant $\Dn^{-\frac{n+1}{2}}$.
Now we note that
$$
  \rho_{C_\z}(0)\cdot \rho_{W_\z}(M) =
  (2\pi)^{-n/2}\, (2\pi)^{-n^2/2}\,
  \exp\Big(-\frac12\|M\|_F^2\Big),
$$
is the density of the standard Gaussian distribution on
$\cM_\R\simeq\R^{n\times (n+1)}$,
so that the last integral (over $M\in\cM_\R$) equals one.
Altogether, we obtain, using $\NJ\Phi_\R =\Dn^{n/2}$,
\begin{equation*}
 \int_{\HdR} \chi\, \rho_{\HdR}\, d\HdR
 = \frac{\Dn^{\frac{n+1}{2}}}{\NJ\Phi_\R}
 = \sqrt{\Dn} . \qedhere
\end{equation*}
\end{proof}

\section{Effective Sampling in the Solution Variety}\label{se:eff-sample}

We turn now to the question of effective sampling in the
solution variety endowed with the measure $\rhozero$
introduced in \S\ref{sec:R&C}.
The goal is to provide the proof of Proposition~\ref{prop:rho2}.


\begin{proposition}\label{cor:crhoW}
In the setting of \S\ref{se:ind-dist} suppose $\oq=0, \s=1$.
Then the pushforward density $\rho_{\cM}$ of $\rho_\hW$ with
respect to $\ps_1\colon W\to\cM$ equals
the standard Gaussian distribution in $\cM$.
The conditional distributions on the fibers of~$\ps_1$
are uniform distributions on unit circles.
Finally, the conditional distribution
on the fibers of $\Psi\colon V\to W$ is induced from the standard Gaussian
in $R_\z$ via the isometry~\eqref{eq:isomaps}.
\end{proposition}

\begin{proof}
Since $\rho_{\Hd}$ is standard Gaussian,
the induced distributions on $C_\z, L_\z$, and $R_\z$
are standard Gaussian as well.
Hence $\rho_{W_\z}$ equals the standard Gaussian
distribution on the fiber $W_\z$.
Moreover, $\rho_{C_\z}(0)=(\sqrt{2\pi})^{-2n}$.
Equation~\eqref{eq:heuri} implies that
$$
  \rho_\cM(M) = \rho_{C_\z}(0)\cdot \rho_{W_\z}(M) =
  (2\pi)^{-n}\, (2\pi)^{-n^2}\,
  \exp\Big(-\frac12\|M\|_F^2\Big),
$$
which equals the density of the standard Gaussian
distribution on $\cM$.

Lemma~\ref{pro:rhoW-2} combined with~\eqref{eq:NJq} gives
$$
 \frac{\rho_W}{\NJ p_1}(M,\z) = \frac1{2\pi}\, \rho_{C_\z}(0)\cdot \rho_{W_\z}(M)
 = \frac1{2\pi}\, \rho_{\cM}(M).
$$
Hence the conditional distributions on the fibers of $p_1$ are uniform.
(Note that this is not true in the case of nonstandard Gaussians.)
The assertion on the conditional distributions on the fibers of $\Psi$
follows from Proposition~\ref{pro:rhoW-1}.
\end{proof}

\begin{proof}[Proof of Proposition~\ref{prop:rho2}]
Proposition~\ref{cor:crhoW} (combined with Lemma~\ref{le:rho1})
shows that the following procedure generates the distribution $\rhozero$.

\begin{enumerate}
\item Choose $M\in\cM$ from the standard Gaussian distribution
(almost surely $M$ has rank~$n$),
\item compute the unique $[\z]\in\proj^n$ such that $M\z=0$,
\item choose a representative~$\z$ uniformly at random
in $[\z]\cap\bS^n$,
\item compute $g_{M,\z}$, cf.~\eqref{eq:isomaps},
\item choose $h\in R_\z$ from the standard Gaussian distribution,
\item compute $q=g_{M,\z} + h$ and return $(q,[\z])$.
\end{enumerate}

An elegant way of choosing $h$ in step~5 is to draw
$f\in\Hd$ from $N(0,\Id)$ and then to compute
the image~$h$ of $f$ under the
orthogonal projection $\H_\z\to R_\z$.
Since the orthogonal projection of a standard Gaussian is a
standard Gaussian, this amounts to draw
$h$ from a standard Gaussian in~$R_\z$.
For computing the projection~$h$ we note that
the orthogonal decomposition $f=k+g_{M,\z} +h$ with $k\in C_\z$,
$M=[m_{ij}]\in\cM$, and $h\in R_\z$ is obtained as
\begin{eqnarray*}
k_i &=& f_i(\z)\langle X,\z\rangle^{d_i}\\
m_{ij} &=& d_i^{-1/2} \big( \partial_{X_j} f_i(\z) - d_i f_i(\z) \bar{\zeta_j} \big)\\
h&=& f-k-g_{M,\z}.
\end{eqnarray*}
(Recall $Dg_{M,\z}(\z) = \Delta M$ and note
$\frac{\partial}{X_j}\langle X,\z\rangle^{d_i}(\z) = d_i\overline{\z}_j$.)

It is easy to check that $\Oh(N)$ samples from the
standard Gaussian distribution on $\R$ are sufficient for
implementing this procedure.
As for the operation count: step~(4)
turns out to be the most expensive one and can be done, e.g.,
as follows. Suppose that all the coefficients of
$\langle X,\z\rangle^{k-1}$ have already been computed.
Then each coefficient of
$
 \langle X,\z\rangle^{k} =
 (X_0\bar{\z}_0 + \cdots + X_n\bar{\z}_n)
 \langle X,\z\rangle^{k-1}
$
can be obtained by $\Oh(n)$ arithmetic operations,
hence all the coefficients of $\langle X,\z\rangle^{k}$
are obtained with $\Oh\big(n{n+k\choose n}\big)$
operations. It follows that
$\langle X,\z\rangle^{d_i}$
can be computed with
$\Oh(d_i n N_i)$ operations,
hence  $\Oh(D n N)$ operations suffice for
the computation of $g_{M,\z}$.
It is clear that this is also an upper bound on the cost
of computing $(q,\z)$.
\end{proof}

\section{Homotopies with a Fixed Extremity}

We provide now the proof of the remaining results stated in Section~\ref{sec:main_results}.
The next two cases we wish to analyze (the condition-based analysis
of \LV\ and a solution for Smale's 17th problem with moderate degrees) share
the feature that one endpoint of the homotopy segment is fixed, not randomized.
This sharing actually allows one to derive both corresponding results
(Theorems~\ref{thm:CBA} and~\ref{thm:S17FD}, respectively) as a
consequence of the following statement.

\begin{theorem}\label{thm:IBA}
For $g\in S(\Hd)\setminus\Sigma$ we have
$$
   \E_{f\in S(\Hd)} \bigg(\dsp(f,g) \int_0^1 \mu_2^2(q_\tau) d\tau\bigg)
   \leq 818\,D^{3/2} N(n+1)\mum^2(g)+ 0.01 .
$$
\end{theorem}

The idea to prove Theorem~\ref{thm:IBA} is simple. For small values of
$\tau$ the system $q_\tau$ is close to $g$ and therefore,
the value of $\mu_2^2(q_\tau)$ can be bounded by
a small multiple of $\mum^2(g)$. For the remaining values of $\tau$,
the corresponding $t=t(\tau)$ is bounded away from 0 and therefore
so is the variance $\sigma_t^2$ in the distribution $N(\oq_t,\sigma_t^2\Id)$
for $q_t$. This allows one to control the denominator in the right-hand side
of Theorem~\ref{th:mu2-bound} when using this result. Here are the
precise details.

In the following fix $g\in S(\Hd)\setminus\Sigma$.
First note that we may again replace the uniform distribution of $f$ on $S(\Hd)$ by
the truncated Gaussian $N_A(0,\Id)$.
As before we chose $A:=\sqrt{2N}$.
We therefore need to bound the quantity
$$
   Q_g:=\E_{f\sim N_A(0,\Id)}\bigg(\dsp(f,g)
       \int_0^1 \mu_2^2 (q_\tau)d\tau\bigg).
$$
To simplify notation, we set as before
$\e=0.13$, $C=0.025$, $\lambda= 6.67\cdot 10^{-3}$,
and define
$$
   \delta_0 := \frac{\lambda}{D^{3/2}\mum^2(g)}, \quad
   t_A:=\frac{1}{1+A+1.00001\,\frac{A}{\delta_0}}.
$$

\begin{proposition}\label{prop:Q}
We have
$$
   Q_g \leq (1+\e)^2 \delta_0\, \mum^2(g)
   + \frac{A}{P_{A,1}} \int_{t_A}^1
    \E_{q_t\sim N(\oq_t,t^2\Id)}
    \bigg(\frac{\mu_2^2(q_t)}{\|q_t\|^2}\bigg)\, dt ,
$$
where $\oq_t = (1-t)g$.
\end{proposition}

\begin{proof}
Let $\z^{(1)},\ldots,\z^{(\Dn)}$ be the zeros of $g$
and denote by
$(q_\tau,\zeta_\tau^{(j)})_{\tau\in[0,1]}$ the
lifting of $E_{f,g}$ in $V$
corresponding to the initial pair $(g,\z^{(j)})$
and final system~$f\in\Hd\setminus\Sigma$.

Equation~\eqref{eq:b1} for $i=0$ in the proof of
Theorem~\ref{thm:main_path_following}
shows the following:
for all $j$ and all
$\tau\leq \frac{\lambda}{\dsp(f,g) D^{3/2}\mun^2(g,\z^{(j)})}$
we have
$$
   \mun(q_\tau,\zeta_\tau^{(j)}) \leq (1+\e)\mun(g,\z^{(j)})
    \leq (1+\e)\mum(g) .
$$
In particular,
this inequality holds for all $j$ and all
$\tau\leq \frac{\delta_0}{\dsp(f,g)}$ and
hence, for all such $\tau$, we have
\begin{equation}\label{eq:tau0}
   \mu_2(q_\tau) \leq (1+\e) \mum(g).
\end{equation}
Splitting the integral in $Q_g$ at
$\tau_0(f):=\min\big\{1,\frac{\d_0}{\dsp(f,g)}\big\}$ we obtain
\begin{eqnarray*}
   Q_g &=& \E_{f\sim N_A(0,\Id)}\Big(\dsp(f,g)
        \int_0^{\tau_0(f)} \mu_2^2 (q_\tau)\,d\tau\Big) \\
        && \quad + \E_{f\sim N_A(0,\Id)}\Big(\dsp(f,g)
       \int_{\tau_0(f)}^1 \mu_2^2 (q_\tau)\,d\tau\Big).
\end{eqnarray*}
Using \eqref{eq:tau0} we bound the first term in the right-hand side as follows,
$$
  \E_{f\sim N_A(0,\Id)}\Big(\dsp(f,g)
       \int_0^{\tau_0(f)} \mu_2^2 (q_\tau)\,d\tau\Big)
    \leq (1+\e)^2\, \d_0 \mum(g)^2.
$$
To bound the second term, we w.lo.g.\ assume that $\tau_0(f) \le 1$.
We apply Proposition~\ref{prop:a-ALH}
to obtain, for a fixed $f$,
$$
   \dsp(f,g)
       \int_{\tau_0(f)}^1 \mu_2^2 (q_\tau)\,d\tau
  \leq \int_{t_0(f)}^1 \|f\| \frac{\mu_2^2 (q_t)}{\|q_t\|^2}\,dt,
$$
where $t_0(f)$ is given by
$$
   t_0(f) =\frac{1} {1+\|f\|(\sin\a\cot\d_0-\cos\a)} ,
   \quad \a:=\dsp(f,g).
$$
Now note that $\|f\|\leq A$ since we draw $f$ from $N_A(0,\Id)$.
This will allow us to bound $t_0(f)$ from below by a quantity independent of~$f$.
For $\|f\|\leq A$ we have
$$
    0 \le \sin\a \cot\delta_0 -\cos\a \le \frac{1}{\sin\delta_0} - \cos\a
    \le \frac{1}{\sin\delta_0} +1
$$
and moreover,
$\sin\delta_0 \ge 0.99999\,\delta_0$
since $\delta_0\le 2^{-3/2}\lambda \le 0.00236$.
We can therefore bound $t_0(f)$ as
$$
  t_0(f) \geq \frac{1}{1+A+\frac{A}{\sin(\delta_0)}}
  \geq \frac{1}{1+A+1.00001\,\frac{A}{\delta_0}}=t_A.
$$
We can now bound the second term in $Q_g$ as follows
\begin{equation*}\begin{split}
\E_{f\sim N_A(0,\Id)}\Big(\dsp(f,g)
      \int_{\tau_0(f)}^1 \mu_2^2 (q_\tau)\,d\tau\Big)
\leq
\E_{f\sim N_A(0,\Id)} \Big( A \int_{t_A}^1
         \frac{\mu_2^2(q_t)}{\|q_t\|^2}\, dt \Big)\\
   =    A \int_{t_A}^1
  \E_{f\sim N_A(0,\Id)}
        \bigg(\frac{\mu_2^2(q_t)}{\|q_t\|^2}\bigg)\, dt
    \leq  \frac{A}{P_{A,1}} \int_{t_A}^1
    \E_{f\sim N(0,\Id)}
    \bigg(\frac{\mu_2^2(q_t)}{\|q_t\|^2}\bigg)\, dt.
\end{split}
\end{equation*}
To conclude, note that, for fixed~$t$ and when $f$ is distributed
following $N(0,\Id)$, the variable
$q_t = (1-t)g + tf$ follows the Gaussian
$N(\oq_t, t^2\Id)$,
where $\og_t=(1-t)g$.
\end{proof}

\begin{proof}[Proof of Theorem~\ref{thm:IBA}]
By homogeneity we can replace the uniform distribution on $S(\Hd)$
by $N_A(0,\Id)$, so that we only need to estimate~$Q_g$ by
the right-hand side of Proposition~\ref{prop:Q}.
In order to bound the first term there we note that
$$
  (1+\e)^2 \delta_0\, \mum^2(g) = (1+\e)^2\lambda D^{-3/2}
  \le  (1+\e)^2\lambda \le 0.01.
$$
For bounding the second term
we apply Theorem~\ref{th:mu2-bound} to deduce that
\begin{equation*}
\begin{split}
    \int_{t_A}^1
    \E_{q_t\sim N(\oq_t,t^2\Id)}
    \Big(\frac{\mu_2^2(q_t)}{\|q_t\|^2}\Big)\, dt
  &  \leq\ \int_{t_A}^1\frac{e(n+1)}{2t^2}\,dt
  =  \frac{e(n+1)}{2}\bigg(\frac1{t_A}-1\bigg) \\
  & = \frac{e(n+1)A}{2}\Big(1+\frac{1.00001}{\delta_0} \Big).
\end{split}
\end{equation*}
Replacing this bound in Proposition~\ref{prop:Q} we obtain
\begin{align*}
   Q_g \; &\leq\;
    \frac{eA^2(n+1)}{2P_{A,1}}\bigg(1+\frac{1.00001}{\lambda}D^{3/2}\mum^2(g) \bigg) + 0.01 \\
  &\leq\;
   2e N (n+1) D^{3/2}\mum^2(g)
   \bigg(\frac1{D^{3/2}} +\frac{1.00001}{\lambda}\bigg)+ 0.01 \\
  &\leq\;
   818\,N(n+1)D^{3/2}\mum^2(g)+ 0.01,
\end{align*}
where we used $D\ge 2$ for the last inequality.
\end{proof}
\smallskip

\subsection{Condition-based Analysis of \LV\ (proof)}\label{sec:cond_based}

\begin{proof}[Proof of Theorem~\ref{thm:CBA}]
The result follows immediately by combining
Proposition~\ref{cor:main_path_following} with Theorem~\ref{thm:IBA},
with the roles of $f$ and $g$ swapped.
\end{proof}
\smallskip

\subsection{The Complexity of a Deterministic Homotopy Continuation}\label{sec:17}

We next prove Theorem~\ref{thm:S17FD},
beginning with some general considerations.
The unitary group $\cU(n+1)$ naturally acts on $\proj^n$ as well as
on $\Hd$ via $(\nu,f)\mapsto f\circ\nu^{-1}$.
The following lemma results from the unitary invariance of our
setting. The proof is immediate.

\begin{lemma}\label{le:Uinv}
Let $g\in\Hd$, $\z\in\proj^n$ be a zero of $g$,
and $\nu\in\cU(n+1)$.
Then
$\mun(g,\z)=\mun(g\circ\nu^{-1},\nu\z)$.
Moreover, for $f\in\Hd$, we have
$K(f,g,\z) = K(f\circ\nu^{-1},g\circ\nu^{-1},\nu\z)$. \eproof
\end{lemma}

Recall $\oU_i = \frac1{\sqrt{2n}}(X_0^{d_i} - X_i^{d_i})$ and
denote by $z_{(i)}$ a $d_i$th primitive root of unity.
The $\Dn$ zeros of $\oU=(\oU_1,\ldots,\oU_n)$
are the points
$\bz_j=\big(1:z_{(1)}^{j_1}:\ldots:z_{(n)}^{j_n}\big)\in\proj^n$
for all the possible tuples $j=(j_1,\ldots,j_n)$ with
$j_i\in\{0,\ldots,d_i-1\}$.
Clearly, each $\bz_j$ can be obtained from
$\bz_1:=(1:1:\ldots:1)$ by a unitary transformation~$\nu_j$,
which leaves $\oU$ invariant, that is,
$$
\nu_j\bz_1 =\bz_j,\quad \oU\circ\nu_j^{-1}=\oU.
$$
Hence Lemma~\ref{le:Uinv} implies $\mun(\oU,\bz_j)=\mun(\oU,\bz_1)$
for all $j$. In particular,
$\mum(\oU) = \mun(\oU,\bz_1)$.

\begin{proposition}\label{prop:inv2}
$K_{\oU}(f)=K(f,\oU,\bz_1)$ satisfies
$$
   \E_{f\in S(\Hd)}K_{\oU}(f) =
   \E_{f\in S(\Hd)} \frac{1}{\Dn}\sum_{j=1}^{\Dn} K(f,\oU,\bz_j).
$$
\end{proposition}

\begin{proof}
Lemma~\ref{le:Uinv} implies for all~$j$
$$
   K(f,\oU,\bz_1) = K(f\circ\nu_j^{-1},\oU\circ\nu_j^{-1},\nu_j\bz_1)
   = K(f\circ\nu_j^{-1},\oU,\bz_j).
$$
It follows that
$$
K_{\oU}(f)=K(f,\oU,\bz_1)=\frac1{\Dn}\sum_{j=1}^{\Dn}
   K(f\circ\nu_j^{-1},\oU,\bz_j).
$$
The assertion follows now since, for all
measurable functions $\varphi\colon S(\Hd)\to\R$ and all $\nu\in\cU(n+1)$,
we have
\begin{equation*}\label{eq:inv1}
  \E_{f\in S(\Hd)}\varphi(f) =
  \E_{f\in S(\Hd)}\varphi(f\circ\nu),
\end{equation*}
due to the isotropy of the uniform measure on $S(\Hd)$,
\end{proof}
\medskip

\begin{lemma}\label{le:comp-munorm}
We have
$$
\mum^2(\oU) \le 2n\, \max_i \frac1{d_i} (n+1)^{d_i-1}\le 2\, (n+1)^{D} .
$$
\end{lemma}

\begin{proof}
Recall $\mum(\oU)=\mun(\oU,\bz_1)$, so it suffices to bound
$\mun(\oU,\bz_1)$.  Consider
$M:=\diag(d_i^{-\frac12} \|\bz_1\|^{1-d_i})\, D\oU(\bz_1)
\in\C^{n\times (n+1)}$.
By definition we have (cf.~\S\ref{se:cond-num})
$$
 \mun(\oU,\bz_1) = \|\oU\|\, \| M^\dagger\|_\op
 = \| M^\dagger\|_\op =\frac{1}{\s_{\min}(M)},
$$
where $\s_{\min}(M)$ denotes the smallest singular value of~$M$.
It can be characterized as  a constrained minimization problem as follows:
$$
      \s_{\min}^2(M) =\min_u \|Mu\|^2_\op \quad\mbox{subject to  }
     u\in (\ker M)^\perp,\ \|u\|^2=1 .
$$
In our situation, $\ker M=\C(1,\ldots,1)$ and $D\oU(\bz_1)$ is given by
the following matrix, shown here for $n=3$:
$$
  D\oU(\bz_1) = \frac1{\sqrt{2n}}
  \begin{bmatrix}
  d_1 & -d_1 & 0 & 0 \\
  d_2 & 0 & -d_2 & 0 \\
  d_3 & 0 & 0 & -d_3 \\
 \end{bmatrix}.
$$
Hence for $u=(u_0,\ldots,u_n)\in\C^{n+1}$,
$$
 \|Mu\|^2 = \frac1{2n} \sum_{i=1}^n  \frac{d_i}{(n+1)^{d_i-1}} |u_i -u_0|^2 \ge
  \frac1{2n} \min_i \frac{d_i}{(n+1)^{d_i-1}} \cdot \sum_{i=1}^n |u_i -u_0|^2 .
$$
A straightforward calculation shows that
$$
  \sum_{i=1}^n |u_i -u_0|^2 \ge 1 \quad
\mbox{if }\quad \sum_{i=0}^n u_i =0,\ \sum_{i=0}^n |u_i|^2 = 1.
$$
The assertion follows by combining these observations.
\end{proof}

\begin{proof}[Proof of Theorem~\ref{thm:S17FD}]
Equation~\eqref{eq:cor3.4} in the proof of
Proposition~\ref{cor:main_path_following}
implies for $g=\oU$ that
$$
 \frac1{\Dn} \sum_{i=1}^\Dn K(f,\oU,\bz_i)
 \le 245\,D^{3/2}\,\dsp(f,\oU)\int_0^1\mu_2^2(q_\tau)\,d\tau.
$$
Using Proposition~\ref{prop:inv2} we get
$$
    \E_{f\in S(\Hd)} K_{\oU}(f)  \le 245\,D^{3/2}\,
    \E_{f\in S(\Hd)} \Big(\dsp(f,\oU)\int_0^1\mu_2^2(q_\tau)\,d\tau \Big) .
$$
Applying Theorem~\ref{thm:IBA} with $g=\oU$ we obtain
$$
   \E_{f\in S(\Hd)} K_{\oU}(f)  \le 245\,D^{3/2}\,
    \big(818\, D^{3/2} N (n+1)\,\mum^2(\oU) + 0.01 \big).
$$
We now plug in the bound
$\mum(\oU)^2 \le 2 (n+1)^D$
of Lemma~\ref{le:comp-munorm} to obtain
$$
   \E_{f\in S(\Hd)} K_{\oU}(f)  \le
     400820\, D^{3}\, N(n+1)^{D+1}  + 2.45\, D^{3/2} .
$$
This is bounded from above by
$400821\, D^{3}\, N(n+1)^{D+1}$,
which completes the proof.
\end{proof}

\section{A near solution to Smale's 17th problem}

We finally proceed with the proof of Theorem~\ref{th:near17}. The
algorithm we will exhibit uses different routines for $D\leq n$
and $D>n$. Our exposition reflects this structure.

\subsection{The case $D\le n$}\label{se:Dlen}

Theorem~\ref{thm:S17FD} bounds the number of iterations of Algorithm \FD\ as
$$
 \E_{f\in S(\Hd)}K_{\oU}(f) = \Oh(D^3 N n^{D+1}).
$$
For comparing the order of magnitude of this upper bound to the input size
$N= \sum_{i=1}^n {n +d_i\choose n}$ we need the following technical lemma
(which will be useful for the case $D>n$ as well).

\begin{lemma}\label{le:B17}
\begin{enumerate}
\item For $D\le n$, $n\ge 4$, we have
$$
 n^D \le {n+D\choose D}^{\ln n} .
$$

\item For $D^2\ge n\ge 1$ we have
$$
 \ln n \leq 2\ln\ln {n+D\choose n}+4.
$$

\item For $0<c<1$ there exists $K$ such that for all $n,D$
$$
   D\le n^{1-c} \Longrightarrow n^D \le {n+D\choose n}^K.
$$

\item For $D\le n$ we have
$$
 n^D \le N^{2\ln\ln N +\Oh(1)} .
$$

\item For $n\le D$ we have
$$
 D^n \le N^{2\ln\ln N +\Oh(1)} .
$$
\end{enumerate}
\end{lemma}

\begin{proof}
Stirling's formula states
$n!=\sqrt{2\pi} n^{n+\frac12} e^{-n} e^{\frac{\Theta_n}{12n}}$
with $0<\Theta_n<1$.
Let $H(x)=x\ln\frac1{x} + (1-x)\ln\frac1{1-x}$ denote the binary entropy function,
defined for $0<x<1$.
By a straightforward calculation we get from Stirling's formula
the following asymptotics for the binomial coefficient:
for any $0<m<n$ we have
\begin{equation}\label{eq:binomial-approx}
    \ln{n\choose m} =
    n H\Big(\frac{m}{n}\Big) + \frac12 \ln\frac{n}{m(n-m)} -1
   + \e_{n,m},
\end{equation}
where $-0.1<\e_{n,m} <0.2$. This formula holds as well for
the extension of binomial coefficients on which $m$ is not necessarily
integer.

(1) The first claim is equivalent to
$e^D \le {n+D\choose D}$.
The latter is easily checked for $D\in\{1,2,3\}$ and $n\ge 4$.
So assume $n\ge D\ge 4$. By monotonicity it suffices to show that
$e^D\le {2D\choose D}$ for $D\ge 4$.
Equation~\eqref{eq:binomial-approx} implies
$$
   \ln {2D\choose D}  > 2D\ln2 +\frac12 \ln\frac{2}{D} -1.1
$$
and the right-hand side is easily checked to be at least $D$,
for $D\ge 4$.

(2) Put $m:=\sqrt{n}$.
If $D\ge m$ then
${n+D\choose n} \ge {n +\lceil m \rceil \choose n}$,
so it is enough to show that
$\ln n \leq 2\ln\ln {n+\lceil m \rceil\choose n}+4$.
Equation~\eqref{eq:binomial-approx} implies
$$
    \ln {n + \lceil m \rceil \choose n} \ge
    \ln {n +  m \choose n} \ge
    (n+m)H\Big(\frac{m}{n+m}\Big) + \frac12 \ln\frac{1}{m} -1.1 .
$$
The entropy function can be bounded as
$$
    H\Big(\frac{m}{n+m}\Big) \ge \frac{m}{n+m} \ln \Big(1+ \frac{n}{m}\Big)
    \ge \frac{m}{n+m}\,\ln m .
$$
It follows that
$$
     \ln {n + \lceil m \rceil\choose n} \ge
     \frac12\sqrt{n}\,\ln n -\frac14 \ln n -1.1
     \ge \frac14\sqrt{n}\,\ln n ,
$$
where the right-hand inequality holds for $n\ge 10$. Hence, for $n \ge 10$,
$$
     \ln \ln {n + \lceil m \rceil\choose n} \ge
     \frac12 \ln n + \ln\ln n -\ln 4 \ge
     \frac12 \ln n - 2 .
$$
This shows the second claim for $n\ge 10$.
The cases $n\le 9$ are easily directly checked.

(3) Writing $D=n\d$ we obtain from Equation~\eqref{eq:binomial-approx}
$$
    \ln {n+D\choose n} =
   (n+D) H\Big( \frac{\d}{1+\d} \Big) -\frac12 \ln D +\Oh(1) .
$$
Estimating the entropy function yields
$$
    H\Big( \frac{\d}{1+\d} \Big) \ge \frac{\d}{1+\d} \ln \Big(1+ \frac{1}{\d}\Big)
    \ge \frac{\d}{2}\, \ln\frac1{\d} = \frac{\d\e}{2}\,\ln n ,
$$
where $\e$ is defined by $\d=n^{-\e}$.
By assumption, $\e\ge c$.
From the last two lines we get
\begin{equation*}
     \frac1{D\ln n}\ln {n+D\choose n}\ \ge\ \frac{c}{2} - \frac{1-c}{2D}
    + \Oh\bigg(\frac1{\ln n}\bigg) .
\end{equation*}
In the case $c\le \frac34$ we have $D \ge n^{1/4}$ and we bound the
above by
$$
    \frac{c}{2} - \frac1{2n^{1/4}} + \Oh\bigg(\frac1{\ln n}\bigg) ,
$$
which is greater than $c/4$ for sufficiently large~$n$.
In the case $c\ge \frac34$ we bound as follows
\begin{equation*}
     \frac1{D\ln n}\ln {n+D\choose n}\ \ge\
     \frac{c}{2} - \frac{1-c}{2}+ \Oh\bigg(\frac1{\ln n}\bigg)
    = c -\frac12 + \Oh\bigg(\frac1{\ln n}\bigg) \ge \frac15
\end{equation*}
for sufficiently large~$n$.

We have shown that for $0<c<1$ there exists $n_c$
such that for $n\ge n_c$, $D\le n^{1-c}$, we have
$$
     n^D \le {n+D\choose n}^{K_c} ,
$$
where $K_c:=\max\{4/c,5\}$.
By increasing $K_c$ we can achieve that the above
inquality holds for all~$n,D$ with $D\le n^{1-c}$.

(4) Clearly, $N\geq {n+D\choose n}$.
If $D\leq\sqrt{n}$ then, by part~(3), there exists~$K$ such that
$$
     n^D \le {n+D\choose n}^K \leq N^K.
$$
Otherwise $D\in[\sqrt{n},n]$ and the desired inequality is an immediate
consequence of  parts (1) and (2).

(5) Use ${n+D\choose n}={n+D\choose D}$ and swap the roles of $n$
and $D$ in part~(4) above.
\end{proof}

Theorem~\ref{thm:S17FD} combined with Lemma~\ref{le:B17}(4) implies that
\begin{equation}\label{eq:smallD}
     \E_{f}K_{\oU}(f) = N^{2\ln\ln N+\Oh(1)} \quad\mbox{ if $D\le n$}.
\end{equation}
Note that this bound is nearly polynomial in~$N$.
Moreover, if $D\le n^{1-c}$ for some fixed $0<c<1$, then
Lemma~\ref{le:B17}(3) implies
\begin{equation}\label{eq:smallD+}
 \E_{f}K_{\oU}(f) = N^{\Oh(1)}.
\end{equation}
In this case, the expected running time is polynomially bounded in
the input size~$N$.

\subsection{The case $D> n$}\label{se:Dgrn}

The homotopy continuation algorithm \FD\ is not efficient for large
degrees---the main problem being that we do not know how to
deterministically compute a starting system~$g$ with small
$\mum(g)$. However, it turns out that an algorithm due to Jim
Renegar~\cite{rene:89}, based on the factorization of the
$u$-resultant, computes approximate zeros and is fast for large degrees.

Before giving the specification of Renegar's algorithm, we need to fix some notation.
We identify $\proj^n_0:=\{(x_0:\cdots:x_n)\in\proj^n \mid x_0\ne 0\}$ with $\C^n$ via 
the bijection $(x_0:\cdots:x_n)\mapsto \underline{x} := (x_1/x_0,\ldots,x_n/x_0)$.
If $x\in\proj^n_0$ we denote by $\|x\|_\aff$ the Euclidean norm of $\underline{x}$, i.e., 
$$
    \|x\|_\aff := \|\underline{x}\| = \Big(\sum_{i=1}^n \Big|\frac{x_i}{x_0}\Big|^2\Big)^{\frac12}
$$
and we put $\|x\|_\aff = \infty$ if $x\in\proj^n\setminus\proj^n_0$.
Furthermore, for $x,y\in\proj^n_0$ we shall write 
$d_\aff(x,y) := \|\underline{x} - \underline{y}\|$ 
and we set $d_\aff(x,y) := \infty$ otherwise. 
An elementary argument shows that
$d_\proj(x,y) \le d_\aff(x,y)$ for $x,y\in\proj^n_0$.

By a {\em $\d$-approximation} of a zero $\z\in\proj^n_0$ of $f\in\Hd$
we understand
an $x\in\proj^n_0$ such that $d_\aff(x,\z) \le \d$.
The following result relates $\d$-approximations to the
approximate zeros in
the sense of Definition~\ref{def:app-zero}.

\begin{proposition}\label{cor:crit}
Let $x$ be a $\d$-approximation of a zero~$\z$ of $f$.
Recall $C=0.025$.
If $D^{3/2}\mun(f,x)\d \le C$,
then $x$ is an approximate zero of $f$.
\end{proposition}

\begin{proof}
We have
$
  d_\proj(x,\z) \le  d_\aff(x,\z) \le \d.
$
Suppose that $D^{3/2}\mun(f,x)\d\le C$. Then,
by Proposition~\ref{prop:apps} with $g=f$, we have
$\mun(f,\z) \le (1+\e)\mun(f,x)$ with $\e=0.13$.
Hence
$$
  D^{3/2}\mun(f,\z)d_\proj(x,\z) \le (1+\e) D^{3/2}\mun(f,x)\d
  \le (1+\e)C.
$$
We have $(1+\e)C \leq u_0 = 3-\sqrt{7}$.
Now use Theorem~\ref{thm:gamma}.
\end{proof}

Consider now $R\ge\d>0$. {\em Renegar's Algorithm} $\ReA(R,\d)$
from~\cite{rene:89}
takes as input $f\in\Hd$ , decides whether
its zero set $V(f)\subseteq \proj^n$ is finite, and if so,
computes $\d$-approximations~$x$ to at least all zeros~$\z$
of~$f$ satisfying $\|\z\|_\aff \le R$.
(The algorithm even finds the multiplicities of those
zeros~$\z$, see \cite{rene:89} for the precise statement.)

Renegar's Algorithm can be formulated in the BSS-model over~$\R$.
Its running time on input~$f$
(the number of arithmetic operations and inequality tests)
is bounded by
\begin{equation}\label{eq:CR}
   \Oh\bigg(n\Dn^4(\log\Dn)\bigg(\log\log\frac{R}{\d}\bigg) +
   n^2\Dn^4 {1+\sum_i d_i \choose n}^4\bigg) .
\end{equation}
To find an approximate zero of $f$ we may use $\ReA(R,\d)$ together with
Proposition~\ref{cor:crit} and iterate
with $R=4^k$ and $\d=2^{-k}$ for $k=1,2,\ldots$ until we are successful. More
precisely, we consider the following algorithm:
\smallskip

\algo
\> Algorithm \IRe\\[2pt]
\>{\bf input} $f\in \Hd$\\[2pt]
\>{\tt for} $k=1,2,\ldots$ {\tt do}\\[2pt]
\>\>{\tt run $\Re(4^k,2^{-k})$ on input~$f$}\\[2pt]
\>\>{\tt for all} $\d$-approximations~$x$ found\\[2pt]
\>\>\>{\tt if} $D^{3/2}\mun(f,x)\d \le C$ {\tt stop}
           {\tt and RETURN} $x$\\[2pt]
\falgo
\smallskip

Let $\Sigma_0:=\Sigma\cup\{f\in\Hd \mid V(f)\cap \proj^n_0 = \emptyset\}$.
It is obvious that \IRe\ stops on inputs $f\not\in\Sigma_0$.
In particular, \IRe\ stops almost surely.

The next result bounds the probability $\Prf$ that the main loop
of \IRe, with parameters $R$ and $\d$, fails
to output an approximate zero for a standard Gaussian
input~$f\in\Hd$ (and given $R,\d$).
We postpone its proof to \S\ref{se:LeF}.

\begin{lemma}\label{le:F17}
We have
$\Prf = \Oh(n^3N^2D^6\Dn\d^4 + n R^{-2})$.
\end{lemma}

Let $T(f)$ denote the running time of algorithm \IRe\ on input~$f$.

\begin{proposition}\label{pro:T17}
We have for standard Gaussian $f\in\Hd$
$$
    \E_{f} T(f) = (nN\Dn)^{\Oh(1)} .
$$
\end{proposition}

\begin{proof}
The probability that \IRe\ stops in the $(k+1)$th loop is bounded above by
the probability $p_k$ that $\Re(4^k,2^{-k})$ fails to produce an
approximate zero. Lemma~\ref{le:F17} tells us that
$$
    p_k = \Oh\big(n^3N^2D^6\Dn\, 16^{-k}\big).
$$
If $A_k$ denotes the running time of the $(k+1)$th loop we conclude
$$
    \E_{f} T(f) \le \sum_{k=0}^\infty A_k p_k .
$$
According to \eqref{eq:CR}, $A_k$ is bounded by
$$
    \Oh\bigg(n\Dn^4(\log\Dn)(\log k) +
    n^2\Dn^4 {1+\sum_i d_i \choose n}^4 + (N+n^3)\Dn\bigg) ,
$$
where the last term accounts for the cost of the tests.
The assertion now follows by distributing the products $A_kp_k$
and using that the series
$\sum_{k\ge 1} 16^{-k}$, and $\sum_{k\ge 1} 16^{-k}\log k$
have finite sums.
\end{proof}

\begin{proof}[Proof of Theorem~\ref{th:near17}]
We use Algorithm \FD\ if $D\le n$ and
Algorithm \IRe\ if $D> n$. We have already shown
(see~\eqref{eq:smallD}, \eqref{eq:smallD+})
that the assertion holds if $D\leq n$.
For the case $D>n$ we use Proposition~\ref{pro:T17}
together with the inequality
$\Dn^{\Oh(1)}\leq D^{\Oh(n)}\leq N^{\Oh(\log\log N)}$
which follows from Lemma~\ref{le:B17}(5).
Moreover, in the case $D\ge n^{1+\e}$,
Lemma~\ref{le:B17}(3) implies
$\Dn\le D^n \le N^{\Oh(1)}$.
\end{proof}

\subsection{Proof of Lemma~\ref{le:F17}}\label{se:LeF}

Let $\cE$ denote the set of $f\in\Hd$ such that there is an $x$ on the output list
of $\ReA(R,\d)$ on input~$f$ that satisfies
$C < D^{3/2}\mun(f,x)\d$. Then
$$
    \Prf \le \Prob_{f\in\Hd}\Big\{\min_{\z\in V(f)} \|\z\|_\aff \ge R \Big\} + \Prob\cE.
$$
Lemma~\ref{le:F17} follows immediately from the following two results.

\begin{lemma}\label{le:R17}
For $R>0$ and standard Gaussian $f\in\Hd$ we have
$$
   \Prob_{f\in\Hd} \big\{ \min_{\z\in V(f)} \|\z\|_\aff \ge R \big\} \le \frac{n}{R^2} .
$$
\end{lemma}

\begin{proof}
Choose $f\in\Hd$ standard Gaussian and pick one of the $\Dn$ zeros
$\z_f^{(1)},\ldots,\z_f^{(\Dn)}$ of~$f$ uniformly at random,
call it~$\z$.
Then the resulting distribution of $(f,\z)$ in $V_\proj$ has the
density~$\rhozero$.
Lemma~\ref{le:rho1} implies that
$\z$~is uniformly distributed in~$\proj^n$.
Therefore,
$$
   \Prob_{f\in\Hd} \big\{ \min_{i} \|\z_f^{(i)}\|_\aff \ge R \big\}
   \le \Prob_{\z\in\proj^n} \big\{ \|\z\|_\aff \ge R \big\} .
$$
To estimate the right-hand side probability we observe that
$$
    \|\z\|_\aff \ge R \Longleftrightarrow d_\proj(\z,\proj^{n-1})
    \le \frac{\pi}{2}-\theta,
$$
where $\theta$ is defined by
$R=\tan\theta$ and $\proj^{n-1}:=\{x\in\proj^n\mid x_0=0\}$.
Therefore,
$$
    \Prob_{\z\in\proj^n} \big\{ \|\z\|_\aff \ge R \big\}
   = \frac{\vol\big\{x\in\proj^n\mid d_\proj(x,\proj^{n-1})
    \le \frac{\pi}{2}-\theta\big\} }{\vol(\proj^n)}.
$$
Due to \cite[Lemma~2.1]{BCL:06a} and using $\vol(\proj^n)=\pi^n/n!$,
this can be bounded by
\begin{equation*}
\begin{split}
       \frac{\vol(\proj^{n-1})\vol(\proj^1)}{\vol(\proj^n)}\,
       \sin^2\bigg(\frac{\pi}{2}-\theta\bigg)
    = n\cos^2\theta = \frac{n}{1+R^2}\le \frac{n}{R^2}. \qedhere
\end{split}
\end{equation*}
\end{proof}

\begin{lemma}\label{le:E17}
We have $\Prob\cE = \Oh(n^3 N^2 D^6 \Dn \d^4)$.
\end{lemma}

\begin{proof}
Assume that $f\in\cE$.
Then, there exist $\z,x\in\proj^n_0$ such that
$f(\z)=0$, $\|\z\|_\aff\leq R$, $d_\aff(\z,x)\leq \d$, $\ReA$ returns~$x$, 
and $D^{3/2}\mun(f,x)\d > C$.

We proceed by cases. Suppose first that
$\d\le\frac{C}{D^{3/2}\mun(f,\z)}$.
Then, by Proposition~\ref{prop:apps},
$$
   (1+\e)^{-1} C < (1+\e)^{-1} D^{3/2} \mun(f,x)\d \le D^{3/2} \mun(f,\z)\d,
$$
hence
$$
   \mum(f)\ge \mun(f,\z) \ge (1+\e)^{-1} C D^{-3/2} \d^{-1}.
$$
If, on the other hand,
$\d >\frac{C}{D^{3/2}\mun(f,\z)}$, then we have
$$
   \mum(f) \ge \mun(f,\z) \ge C D^{-3/2} \d^{-1} .
$$
Therefore, for any $f\in\cE$,
$$
   \mum(f) \ge (1+\e)^{-1}C  D^{-3/2} \d^{-1} .  
$$
Theorem~C of \cite{Bez2} states that
$\Prob_f \{\mum(f) \ge \rho^{-1} \} = \Oh(n^3 N^2 \Dn\rho^4)$
for all $\rho>0$. Therefore, we get
$$
    \Prob\cE \leq \Prob_{f\in\Hd}\big\{\mum(f)\geq  (1+\e)^{-1} C D^{-3/2} \d^{-1}\big\}
  = \Oh(n^3 N^2\Dn D^6 \d^4)
$$
as claimed.
\end{proof}

\subsection*{Note added in proof.}

Since the posting of this manuscript
on September 2009,
at {\tt arXiv:0909.2114},
a number of
references have been added to the literature.
The non constructive character of the
main result in~\cite{Bez6} --- the bound
in~\eqref{eq:integral_mu2}---
had also been noticed by Carlos Beltr\'an. 
In a recent paper 
(``A continuation method to solve polynomial systems, and its complexity'',
{\em Numerische Mathematik} 117(1):89-113, 2011)
Beltr\'an proves a
very general constructive version of this result. Our
Theorem~\ref{thm:main_path_following} can be seen as
a particular case (with a correspondingly shorter proof)
of Beltr\'an's paper main result. We understand that
yet another constructive version for the bound
in~\eqref{eq:integral_mu2} is the subject of a paper
in preparation by J.-P.~Dedieu, G.~Malajovich,
and M.~Shub.

Also, Beltr\'an and Pardo have recently rewritten their
paper~\cite{BePa08b} (``Fast linear homotopy to find
approximate zeros of polynomial systems'', 
{\em Found. Comput. Math.}, 11(1):95--129, 2011.)
This revised version,
which increases the length of the manuscript by a factor of
about three, adds considerable detail to a number of issues
only briefly sketched in~\cite{BePa08b}. In particular,
the effective sampling from the solution variety is now
given a full description (which is slightly different to
the one we give in Section~\ref{se:eff-sample}).

\bibliographystyle{plain}
{\small

}

\end{document}